\documentclass[11pt]{amsart}

\usepackage[usenames,dvipsnames]{xcolor}
\usepackage{amsmath,amsthm,verbatim,amssymb,amsfonts,amscd, graphicx, enumerate, enumitem,hyperref}
\usepackage{graphics}
\usepackage{tikz-cd}
\usepackage{pinlabel}
\usepackage{stackrel}
\usepackage{mathtools}
\usepackage{wrapfig}
\usepackage[font=small]{caption}

\definecolor{linkpurple}{RGB}{136,92,150}

\hypersetup{colorlinks=true,linkcolor=linkpurple,citecolor=linkpurple, urlcolor=black}

\usepackage[margin=1.39in,headsep=.3in]{geometry}

\newtheorem{theorem}{Theorem}[section]
\newtheorem{corollary}[theorem]{Corollary}
\newtheorem{lemma}[theorem]{Lemma}
\newtheorem{proposition}[theorem]{Proposition}

\newtheorem*{question*}{Question} 
 
\theoremstyle{definition}

\newtheorem{remark}[theorem]{Remark}

\theoremstyle{definition}

\newtheorem{example}[theorem]{Example}

\newcommand{\leftrarrows}{\mathrel{\raise.75ex\hbox{\oalign{%
  $\scriptstyle\leftarrow$\cr
  \vrule width0pt height.5ex$\hfil\scriptstyle\relbar$\cr}}}}
\newcommand{\lrightarrows}{\mathrel{\raise.75ex\hbox{\oalign{%
  $\scriptstyle\relbar$\hfil\cr
  $\scriptstyle\vrule width0pt height.5ex\smash\rightarrow$\cr}}}}
\newcommand{\Rrelbar}{\mathrel{\raise.75ex\hbox{\oalign{%
  $\scriptstyle\relbar$\cr
  \vrule width0pt height.5ex$\scriptstyle\relbar$}}}}

\usepackage[utf8]{inputenc}

\makeatletter
\def\leftrightarrowsfill@{\arrowfill@\leftrarrows\Rrelbar\lrightarrows}
\newcommand{\xleftrightarrows}[2][]{\ext@arrow 3399\leftrightarrowsfill@{#1}{#2}}
\makeatother

\definecolor{MutedBlue}{RGB}{40,90,160}

\definecolor{violet}{rgb}{.6,0,.6}
\definecolor{green}{rgb}{.0,.8,0}
\definecolor{darkred}{rgb}{.7,0,0}
\definecolor{darkyellow}{rgb}{.4,.4,0}
\definecolor{orangenew}{rgb}{.69,.37,.25}
\definecolor{MutedRed}{RGB}{200,55,55}

\newcommand{\Z}{\mathbb{Z}}
\newcommand{\R}{\mathbb{R}}

\newcommand{\sym}{\operatorname{Sym}}
\newcommand{\diff}{\operatorname{Diff}}

\let\int\relax
\newcommand{\int}{\mathring}

\DeclareMathOperator{\ckh}{\mathcal{C}{Kh}}
\DeclareMathOperator{\id}{{id}}

\DeclareMathOperator{\pt}{{pt}}

\DeclareMathOperator{\kh}{{Kh}}

\DeclareMathOperator{\wh}{Wh}
\DeclareMathOperator{\hfk}{\mathit{HFK}}

\DeclareMathOperator{\sw}{SW}

\usepackage{fancyhdr}
\usepackage[affil-it]{authblk} 
\usepackage{etoolbox}
\usepackage{lmodern}
\makeatletter
\patchcmd{\@maketitle}{\LARGE \@title}{\fontsize{16}{19.2}\selectfont\@title}{}{}
\makeatother

\author[K.\ Hayden]{Kyle Hayden}
\address{Department of Mathematics and Computer Science, Rutgers University--Newark,\linebreak Newark, NJ 07102 USA}
\email{kyle.hayden@rutgers.edu}

\author[S.\ Kim]{Seungwon Kim}
\address{Department of Mathematics, Sungkyunkwan University, Suwon 16419, Republic of Korea}
\email{math751@gmail.com}

\author[M.\ Miller]{Maggie Miller}
\address{Department of Mathematics, Stanford University, Stanford, CA 94305 USA}
\email{maggie.miller.math@gmail.com}

\author[J.\ Park]{\\JungHwan Park}
\address{Department of Mathematical Sciences, Korea Advanced Institute of Science and Technology, Daejeon 34141, Republic of Korea}
\email{jungpark0817@kaist.ac.kr}

\author[I.\ Sundberg]{Isaac Sundberg}
\address{Max Planck Institute for Mathematics, Bonn 53111, Germany}
\email{isaacsundbe@gmail.com}

\title[Seifert surfaces in the 4-ball]{Seifert surfaces in the 4-ball}

\thanks{\tiny KH is supported by NSF grant DMS-2114837. SK was supported by National Research Foundation of Korea (NRF) grants funded by the Korea government (MSIT) (No.\ 2020R1A5A1016126, 2022R1C1C2004559). MM is supported by a fellowship from the Clay Mathematics Institute and partially by a Stanford Science Fellowship. JP is partially supported by Samsung Science and Technology Foundation (SSTF-BA2102-02) and the POSCO TJ Park Science Fellowship. \vspace{-.2in}}

\begin{document}

\vspace*{-.2in}
\maketitle

\vspace{-.125in}

\begin{center}\small
\textsc{Kyle Hayden, Seungwon Kim, Maggie Miller}

\vspace{.015in}

\textsc{JungHwan Park, and Isaac Sundberg}
\end{center}

\begin{abstract}
We answer a question of Livingston from 1982 by producing Seifert surfaces of the same genus for a knot in $S^3$ that do not become isotopic when their interiors are pushed into $B^4$. In particular, we identify examples where the surfaces are not even topologically isotopic in $B^4$, examples that are topologically but not smoothly isotopic, and examples of infinite families of surfaces that are  distinct only up to isotopy rel.\ boundary.  Our main proofs distinguish surfaces using the cobordism maps on Khovanov homology, and our calculations demonstrate the stability and computability of these maps under certain satellite operations. 
\end{abstract}

\vspace{.125in}

\section{Introduction}\label{sec:intro}
\thispagestyle{empty}

\vspace{.05in}

The following question was posed by Livingston \cite{livingston}, generalizing Problem 1.20(C) in \cite{kirby-old,kirby}.

\begin{question*}
    Let $\Sigma_0$ and $\Sigma_1$ be Seifert surfaces of equal genus for a knot $K$. After pushing the interiors of both surfaces into $B^4$, are $\Sigma_0$ and $\Sigma_1$ isotopic  in $B^4$?
\end{question*}

There are many known examples of knots bounding genus-$g$ incompressible Seifert surfaces that are not isotopic in $S^3$, dating back at least to work of Trotter \cite{trotter} (c.f.~\cite{alford}). However, as discussed below, most examples of such Seifert surfaces are known to become isotopic when their interiors are pushed into $B^4$. This contrasts with the codimension-two setting, where there are a wealth of constructions that produce 4-dimensionally knotted surfaces in $B^4$ and other 4-manifolds (e.g.\ \cite{artin,zeeman,fintushel-stern:surfaces}). In this light, Livingston's question asks if 4-dimensional knotting fundamentally requires all four dimensions.  Our main results answer this question.

\begin{theorem}\label{thm:ambient}
There exist minimal genus Seifert surfaces in $S^3$ with the same boundary that are not ambiently isotopic in $B^4$. 
\end{theorem}

We provide a variety of counterexamples arising from three fundamentally distinct constructions, each highlighting different aspects of the problem; see \S\ref{sec:constructions}.  For a  core pair of examples $\Sigma_0$ and $\Sigma_1$,  see Figure~\ref{fig:46band}. Our main proofs distinguish surfaces up to smooth isotopy using the cobordism maps on Khovanov homology. These arguments expand on the computational techniques developed in \cite{sundberg-swann,hayden-sundberg}, including the study of satellite operations (such as Whitehead doubling) on these cobordism maps. In particular, we are able to detect exotically knotted pairs of Seifert surfaces:

\begin{theorem}\label{thm:exotic}
There exist infinitely many distinct pairs of minimal genus Seifert surfaces in $S^3$ that are topologically isotopic rel.\ boundary in $B^4$ yet are not smoothly ambiently isotopic in $B^4$.
\end{theorem} 

 We also show that invariants of branched covers can distinguish Seifert surfaces up to isotopy in $B^4$. Indeed, we have chosen the core examples $\Sigma_0$ and $\Sigma_1$ in such a way that they can also be distinguished using the intersection forms on their double branched covers; this implies that they are not even \textsl{topologically} isotopic in $B^4$.

\begin{theorem}\label{thm:top}
The Seifert surfaces $\Sigma_0$ and $\Sigma_1$ are not topologically isotopic in $B^4$.
\end{theorem}

\begin{figure}
    \centering
           \labellist
\pinlabel $\Sigma_0$ at 10 140
\pinlabel $\Sigma_1$ at 215 140
\endlabellist
    \includegraphics[width=.7\linewidth]{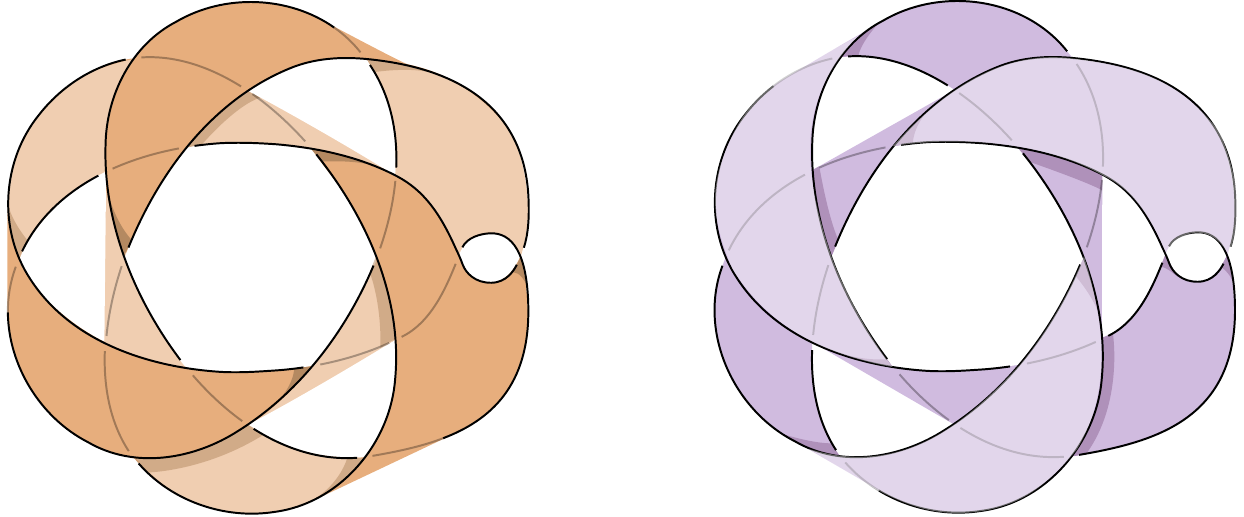}
    \caption{Two genus-1 Seifert surfaces $\Sigma_0$ (left) and $\Sigma_1$ (right) for the same knot $K$ that are not isotopic 
    even when their interiors are pushed into $B^4$.
    }
    \label{fig:46band}

\end{figure}

 We also consider the question of whether knots and links can bound infinitely many Seifert surfaces up to isotopy in $B^4$. Constraints from normal surface theory imply that hyperbolic knots and links only bound finitely many Seifert surfaces of fixed genus up to  isotopy rel.~boundary in $S^3$ \cite{schubert-soltsien} (c.f.~\cite{wilson} and \cite[Theorem~4]{johnson-pelayo-wilson}). Indeed, for any knot or link, there is an essentially unique way to generate infinite families of Seifert surfaces with the same genus and boundary, consisting of taking Haken sums of an initial Seifert surface with an incompressible torus in the knot or link complement. For  knots and links with toroidal complements, the finiteness question is more subtle and depends on whether isotopies are taken rel.~boundary. 
 We show that this subtlety persists even up to isotopy through $B^4$.

\begin{theorem}\label{thm:infinite}
There exist links in $S^3$ that bound an infinite family of connected Seifert surfaces that are freely ambiently isotopic in $S^3$ yet are pairwise distinct up to smooth isotopy rel.~boundary in $B^4$. 
\end{theorem}

We distinguish the examples in Theorem~\ref{thm:infinite} by studying the relative Seiberg-Witten invariants of their branched covers. We note that simpler obstructions suffice in the case of \textsl{disconnected} spanning surfaces. Indeed, as a point of contrast with Livingston's result  \cite{livingston} (which states that {\emph{connected}} Seifert surfaces for an unlink that have the same homeomorphism type are isotopic rel.\ boundary in $B^4$), we show that  unlinks can bound infinitely many disconnected planar surfaces in $S^3$ up to topological isotopy rel.~boundary in $B^4$.

\begin{theorem}\label{thm:infiniteunlink}
The 3-component unlink bounds infinitely many disconnected Seifert surfaces  (each comprised of a disk and an annulus) that are distinct up to topological isotopy rel.~boundary in $B^4$. 
\end{theorem}

As mentioned above, our smooth obstructions hinge on new developments for calculating the cobordism maps on Khovanov homology. This reveals several surprising features of these cobordism maps, including their behavior under certain satellite constructions, as well as certain band sums that allow us to produce minimal genus Seifert surfaces. As an illustration of this, we note that Whitehead doubling plays particularly well with these cobordism maps, enabling us to prove the following general result:

\begin{theorem}\label{thm:sqp}
If $J$ is a nontrivial strongly quasipositive knot, then $\wh(J)$ bounds Seifert surfaces of equal genus that are topologically isotopic rel.~boundary in $B^4$ yet are not smoothly ambiently isotopic in $B^4$.
\end{theorem}

It is interesting to compare these tools from Khovanov homology with  their counterparts in Floer theory. While it seems likely that one could use the cobordism maps in knot Floer homology  $\widehat{\hfk}$ to distinguish pairs of minimal genus Seifert surfaces (perhaps $\Sigma_0$ and $\Sigma_1$), these maps do not appear to be as amenable to direct calculation as the cobordism maps in Khovanov homology.  Moreover, given that we use elementary methods to show that $\Sigma_0$ and $\Sigma_1$  are not even topologically isotopic, it would be more interesting to see the $\widehat{\hfk}$ cobordism maps used to distinguish a pair of Seifert surfaces that are topologically isotopic (or are at least not easily distinguished topologically). Distinguishing non-minimal genus Seifert surfaces appears to be especially subtle, given the close relationship between $\smash{\widehat{\hfk}}$ and the Seifert genus; we discuss this further at the end of Section~\ref{section:khovanov}.

\subsection*{Organization} 
In Section~\ref{sec:constructions}, we provide an overview of our underlying constructions of Seifert surfaces.    In Section~\ref{section:khovanov}, we perform our primary  computations of  Khovanov cobordism maps, distinguishing the core examples $\Sigma_0$ and $\Sigma_1$, as well as their Whitehead doubles $\wh(\Sigma_0)$ and $\wh(\Sigma_1)$. We generalize aspects of these concrete calculations by proving Theorem~\ref{thm:sqp} in \S\ref{subsection:sqp}. In \S\ref{subsec:furtherexamples}, we modify the earlier  examples to produce minimal genus Seifert surfaces that are not smoothly ambiently isotopic, concluding the proof of  Theorem~\ref{thm:exotic}. Finally, in Section~\ref{sec:branched}, we turn to obstructions arising from branched covers, proving Theorems~\ref{thm:top} and~\ref{thm:infinite}.

    \subsection*{Conventions} In this paper, we will always specify whether results hold in the smooth or locally flat category, as we work in each at different times. We will forever refer to the surfaces in Figure~\ref{fig:46band} as $\Sigma_0$ and $\Sigma_1$, and their boundary as $K$.  For convenience, we work with Khovanov homology over $\Z_2$ coefficients.

\subsection*{Acknowledgements} 
Some of the authors learned of the motivating question of this paper from Peter Teichner in spring 2021; we thank him for many interesting conversations and helpful comments. 
This project came about during the ``Braids in Low-Dimensional Topology" conference at ICERM in April 2022. We thank the organizers for putting together an engaging conference.

\section{Constructions and counterexamples}\label{sec:constructions}

 For context, we begin by revisiting some of the historical constructions of pairs of distinct Seifert surfaces with the same genus and boundary in \S\ref{subsec:history}. We then present our three main constructions in \S\ref{subsec:cut}-\ref{subsec:satellite}. These sections state more precise versions of the results from \S\ref{sec:intro}, whose proofs will then be given in \S\ref{section:khovanov}-\S\ref{sec:branched}.

\subsection{Historical constructions.} \label{subsec:history}
There are many known examples of knots bounding genus-$g$ incompressible Seifert surfaces that are not isotopic in $S^3$ (e.g.\  \cite{alford, schaufele, daigle, lyon, trotter, eisner, parris, gabaidetecting, kobayashi, kakimizu, HJS, vafaee}). Most of these examples of distinct Seifert surfaces are known to become isotopic when their interiors are pushed into $B^4$.  The pairs of surfaces of \cite{alford,schaufele,daigle} are constructed by choosing one of two tubes to include in the surface that become isotopic when pushed deeper into $B^4$ (see \cite[p123]{rolfsen}); the pretzel surfaces of \cite{trotter,parris} (c.f.\ \cite{kobayashipretzel}) were shown to be isotopic in $B^4$ by Livingston \cite[Section 6]{livingston}, who also showed that any connected, genus-$g$ Seifert surfaces for an unlink become isotopic in $B^4$; the surface families in \cite{eisner,kakimizu} are obtained by twisting a surface many times about a satellite torus, which can be undone by isotopy in $B^4$ (see Proposition~\ref{prop:twistisotopy}); the examples of \cite{gabaidetecting,kobayashi,HJS,vafaee} arise from a Murasugi sum construction and are isotopic in $B^4$ via an isotopy involving pushing the plumbing region deeper into $B^4$ (see \cite[Remark 3.1]{vafaee}).

However, for one fairly well-known family of examples due to Lyon \cite{lyon}, there is no clear isotopy between the surfaces even when they are pushed into $B^4$. Lyon's examples are 
genus-1 Seifert surfaces obtained by attaching a common band to each of two different choices of annuli with the same boundary. Together, those underlying annuli form  the standard torus in $S^3$, with their common boundary given by an unoriented $(6,-8)$ torus link. Lyon distinguished these surfaces in $S^3$ via the fundamental groups of their complements; Altman \cite{altman} studied the same surfaces and showed that  they are also distinguished by the sutured Floer polytopes of their complements.

\subsection{Cutting up closed surfaces.}\label{subsec:cut}
Our first approach is based on a generalization of Lyon's construction, and is animated by a simple observation: A link $L \subset S^3$ bounds a pair of genus-$g$ Seifert surfaces with disjoint interiors if and only if $L$ lies on a closed genus-$2g$ surface in $S^3$ and separates it into a pair of genus-$g$ surfaces. 

\begin{example} The core examples $\Sigma_0$ and $\Sigma_1$ from Figure~\ref{fig:46band} provide pair of \mbox{genus-1} Seifert surfaces for a twisted Whitehead double $K$ of the left-handed trefoil. As illustrated in Figure~\ref{fig:union}, the union of $\Sigma_0$ and $\Sigma_1$ is a standard, closed, genus-2 surface in $S^3$. (Here the underlying annuli are bounded by an unoriented $(4,-6)$ torus link, and we attach a common band to the right-hand side of each annulus.)
\end{example}

\begin{figure}[h]
\centering
  \captionsetup{width=.92\linewidth}
  \includegraphics[height=36mm]{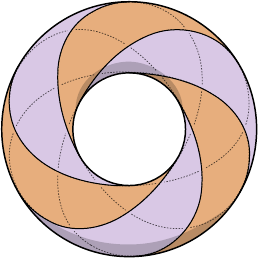}
  \hspace{.5in}
  \includegraphics[height=36mm]{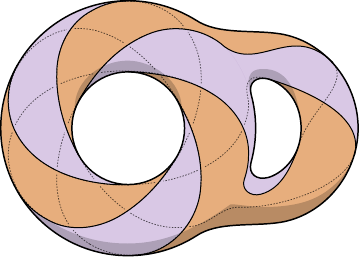}
  \caption{Left: two annuli bounded by an unoriented $(4,-6)$ torus link whose union is a standard torus. Right: the surfaces $\Sigma_0,\Sigma_1$, whose union is a standard genus-2 surface.}
  \label{fig:union}
  
\end{figure}

We distinguish $\Sigma_0$ and $\Sigma_1$ using Khovanov homology in the proof of Proposition~\ref{prop:khovanov}. In Theorem~\ref{thm:doublecover}, we further distinguish $\Sigma_0$ and $\Sigma_1$ up to topological isotopy by showing their double branched covers are not homeomorphic. 

\begin{example}
 A second example of this form is shown in Figure~\ref{fig:wh-tref-embed}, which depicts a genus-4 surface cut into a pair of genus-2 Seifert surfaces for the Whitehead double of the right-handed trefoil knot; after isotopy, these are redrawn in Figure~\ref{fig:wh-tref}. These examples turn out to be topologically isotopic yet smoothly distinct (by Theorem~\ref{thm:sqp}).
\end{example}

\begin{figure}
\center
\includegraphics[width=\linewidth]{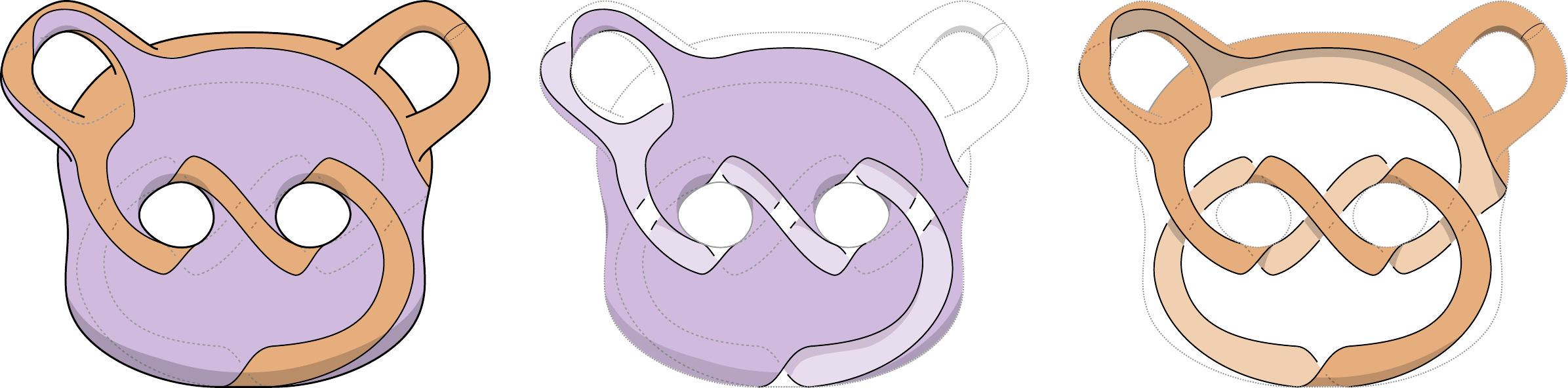}
\caption{Cutting a standard genus-4 surface into a pair of genus-2 Seifert surfaces for the Whitehead double of the right-handed trefoil.}
\label{fig:wh-tref-embed}
\end{figure}

\subsection{Haken sums and torus twists.}\label{subsec:spinning} Using normal surface theory, Schubert-Soltsien \cite{schubert-soltsien} showed that non-satellite knots in $S^3$ have only finitely many  Seifert surfaces of fixed genus up to isotopy rel.~boundary   (known as \emph{strong equivalence}), and the same holds for links with atoroidal complements  \cite[Theorem~4]{johnson-pelayo-wilson}. In contrast, satellite knots and links may have infinitely many Seifert surfaces up to isotopy rel.~boundary in $S^3$. In this case, infinite families may be generated by fixing an initial Seifert surface $S$ and taking Haken sums of $S$ with copies of an incompressible torus $T$ in the knot or link complement $S^3 \setminus \partial S$.

An important example of this operation is given by \emph{torus twists}, defined as follows: Suppose that $S$ intersects $T$ transversely along a collection of essential curves in $T$, all of which will necessarily be parallel.  Dehn's lemma implies that $T$ bounds a solid torus $V$ in $S^3$, so we may choose coordinates identifying $T$ with $T^2=S^1 \times S^1$  such that each curve $S^1 \times \{\text{pt}\}$ bounds a disk in $V$. (The other coordinate direction $\{\text{pt}\}\times S^1$ may be chosen freely, e.g., to be nullhomologous in $S^3 \setminus \mathring{V}$, or to coincide with the curves $S \cap T$.) Extend these coordinates to a neighborhood $T^2 \times [0,1]$ of $T=T^2 \times \{1/2\}$. We define a \emph{meridional twist along $T$} to be a diffeomorphism that is defined on $T^2 \times [0,1]$ by $(x,y,t)\mapsto(x+2\pi t,y,t)$ and by the identity on the rest of $S^3$; a \emph{longitudinal twist} is defined analogously.

For a schematic depiction of the effect of a meridional twist on $S$,  see Figure~\ref{fig:roll_intro}. It is straightforward to show that Seifert surfaces related by torus twists are freely isotopic in $S^3$. (The isotopy can be constructed directly using the fact that $T$ bounds a solid torus in $S^3$, c.f.~\cite[p332]{eisner}.) However, they need not be isotopic rel.~boundary.

\begin{figure}
  \labellist
\hair 2pt
\pinlabel $S$ at 51 25
\pinlabel $T$ at 200 25
\pinlabel $S'$ at 346 25
\endlabellist
\center
\includegraphics[width=.825\linewidth]{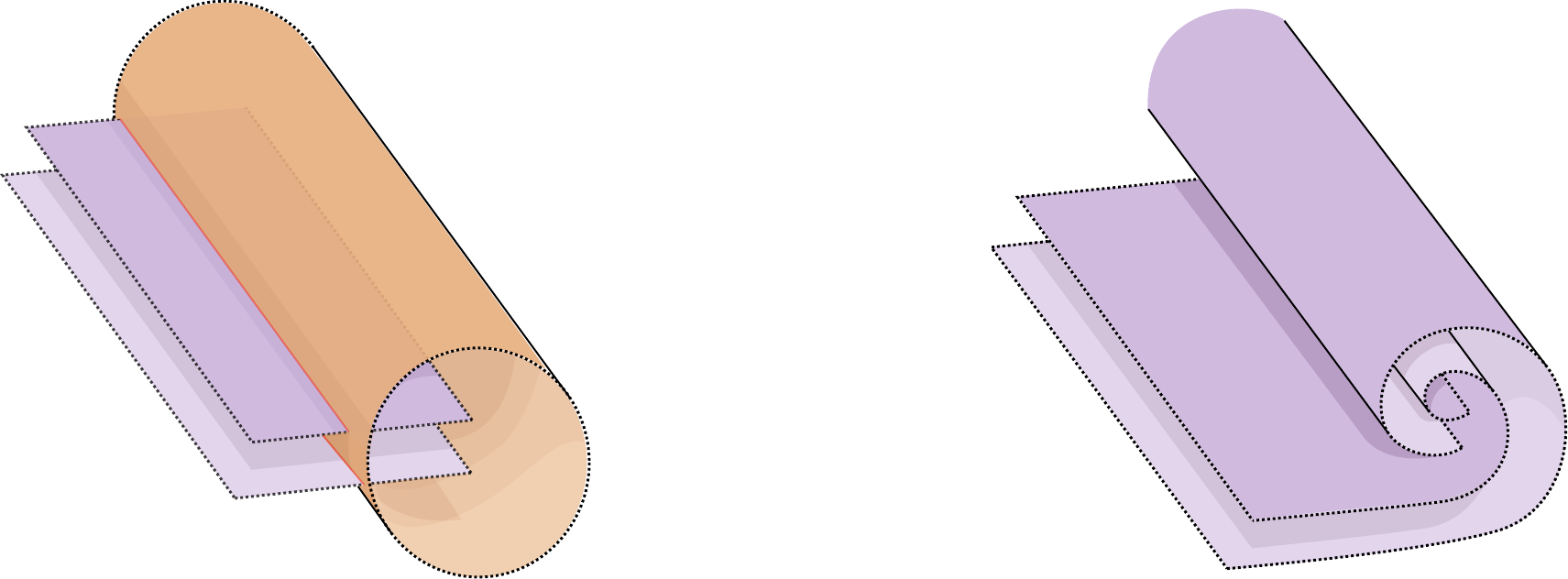}
\caption{Twisting a surface $S$ along a torus $T$ producing a new surface $S'$, where $T$ intersects $S$ transversely in its interior.}\label{fig:roll_intro}
\end{figure}

\begin{example}\label{ex:infinite}
Consider the surface $S_0$ in Figure~\ref{fig:infinite}; it is constructed by taking a pair of interlocked annuli bounded by an unoriented $(4,8)$-torus link, then joining these annuli by a twisted band. Its boundary $L=\partial S_0$ is a 3-component link, and the link complement $S^3 \setminus L$ contains an essential torus $T$ meeting $S_0$ transversely along a pair of simple closed curves. Twisting $S_0$ along $T$ yields an infinite family of surfaces $S_n$ for $L$. In Section~\ref{sec:branched}, we will show that the Seifert surfaces $S_n$ are not smoothly isotopic rel.~boundary in $B^4$ (proving Theorem \ref{thm:infinite}).
\end{example}

\begin{figure}
\centering
  \captionsetup{width=.9\linewidth}
  \includegraphics[width=\linewidth]{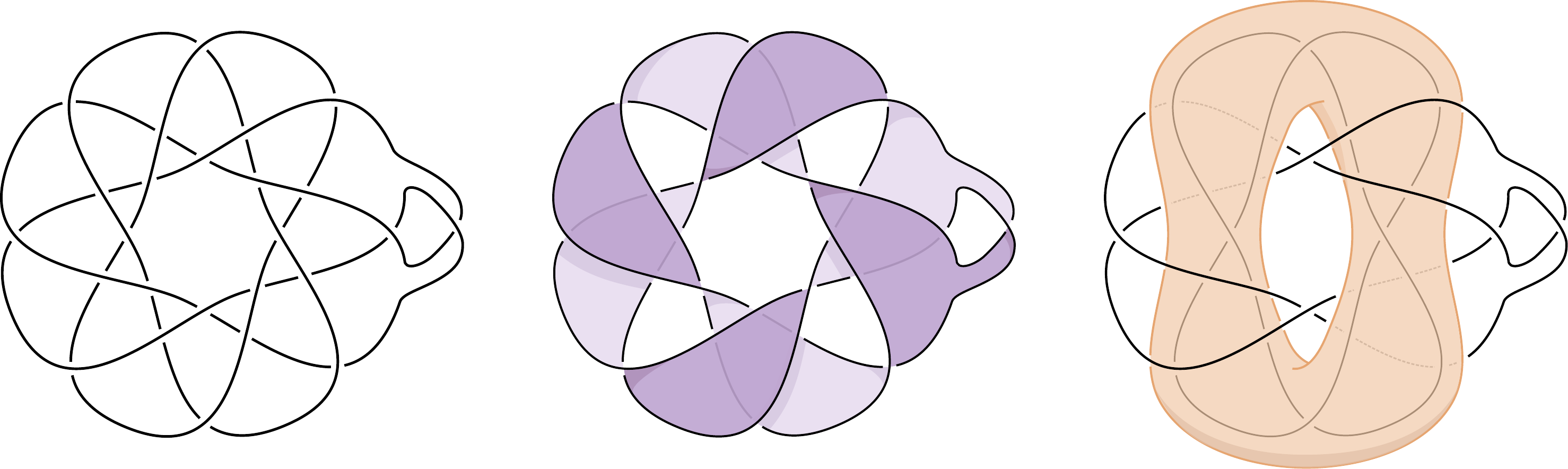}
  \includegraphics[width=45mm]{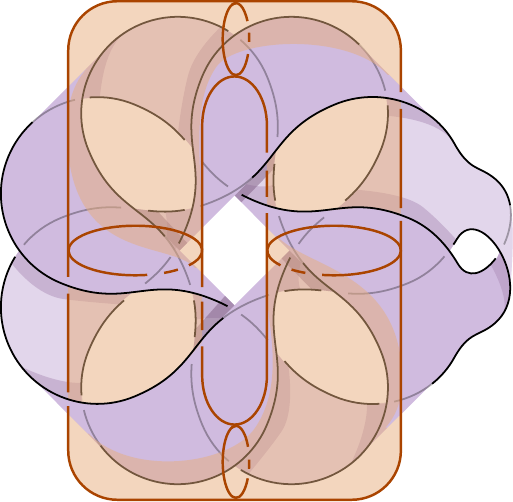}
  \caption{{Twisting the top middle surface about the torus featured in the top right yields an infinite family of Seifert surfaces for a 3-component link that are not isotopic rel.\ boundary. See Example \ref{ex:infinite} and Theorem \ref{thm:infinite}.}}
  \label{fig:infinite}
\end{figure}

\subsection{Satellite surfaces}\label{subsec:satellite}
 Satellite operations provide an important construction of knots in 3-manifolds, and analogous constructions can be applied to Seifert surfaces. Let $L$ be a link in $S^3$ whose complement $S^3 \setminus L$ contains an incompressible torus $T$. We say that a spanning surface $S \subset S^3$ for $L$ is a \emph{satellite with respect to $T$} if $T$ is transverse to $S$ and bounds a solid torus $V$ such that: (i) $L=\partial S$ lies in $V$, and (ii) $S$ meets $S^3 \setminus \mathring{V}$ in a collection of parallel push-offs of a fixed surface $F$ bounded by a longitude on $\partial V$. Technically, we allow $F$ to be empty. Note that if $F$ is nonempty, then it must have positive genus, otherwise the torus $T$ would be compressible. 

In practice, we will construct such surfaces by fixing an initial Seifert surface $F$ for a knot $J$ and a tubular neighborhood $V$ of $J=\partial F$, then choosing $n$ parallel push-offs of $F$ and gluing them to a ``pattern surface'' lying entirely in $V$ (for examples, see Figure~\ref{fig:satellite}). The resulting surface $S$ is a satellite with respect to $T=\partial V$.

\begin{figure}
    \includegraphics[width=.925\linewidth]{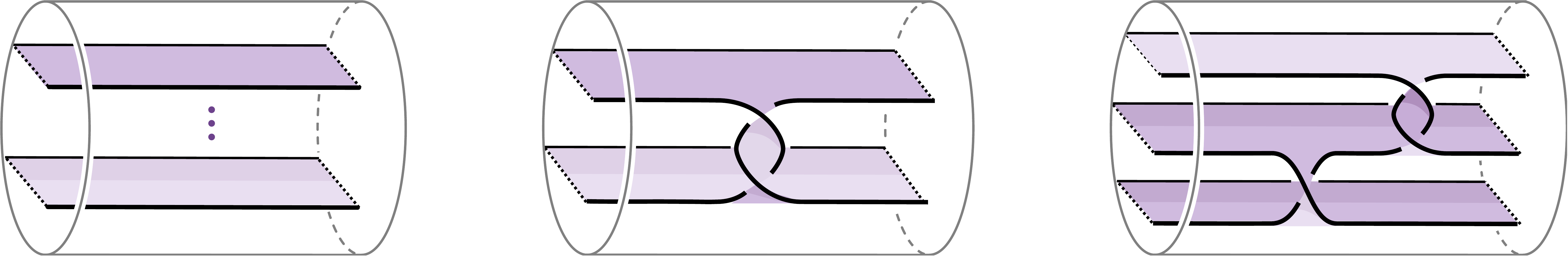}
    
    \captionsetup{width=.9\linewidth}
    \caption{From left to right, satellite pattern surfaces bounded by the $n$-copy pattern,  the positive Whitehead pattern, and  the Mazur pattern.}

    \label{fig:satellite}
\end{figure}

\begin{example}\label{ex:WhTref}
Let $J$ denote the right-handed trefoil, and let $\wh(J)$ denote its positively clasped, untwisted Whitehead double, illustrated on the left side of Figure~\ref{fig:wh-tref}. The middle of Figure~\ref{fig:wh-tref} depicts a genus-2 satellite Seifert surface for $\wh(J)$ built from  two copies of the standard fiber surface for $J$. We can produce a second genus-2 Seifert surface for $\wh(J)$ by stabilizing the standard genus-1 Seifert surface for $\wh(J)$, as shown on the right side of Figure~\ref{fig:wh-tref}. (Technically, this is also a satellite with respect to the natural torus, where the subsurface $F$ is empty.)   We will distinguish these surfaces up to smooth isotopy in $B^4$ using Khovanov homology in \S\ref{subsection:sqp}. However, by Theorem~\ref{thm:conwaypowell} below, these surfaces turn out to be topologically isotopic rel.~boundary in $B^4$.
\end{example}

\begin{figure}
\center
\includegraphics[width=\linewidth]{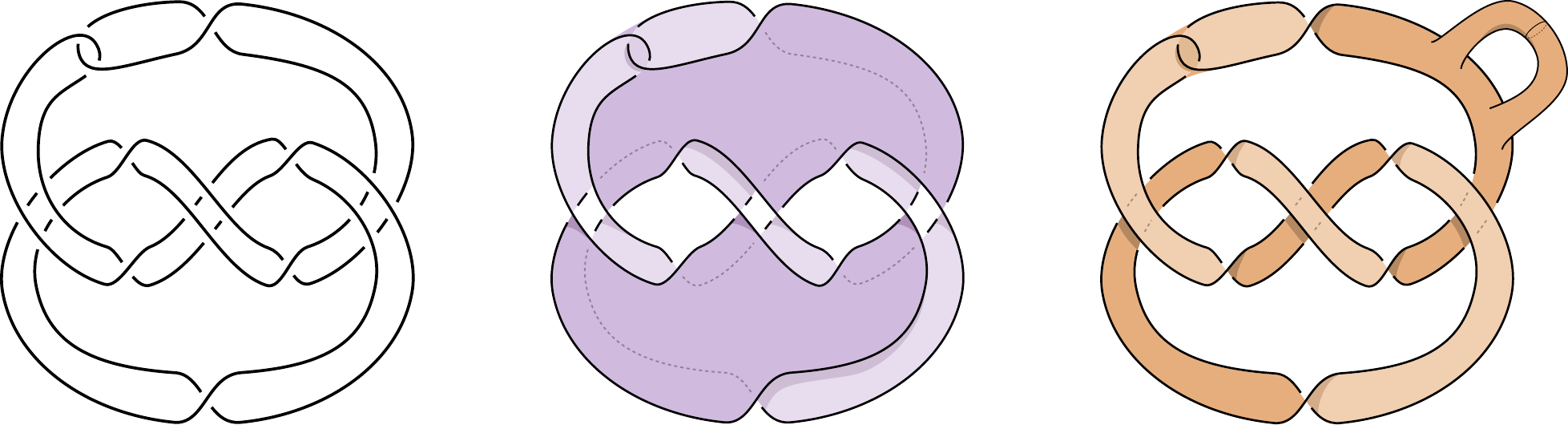}
\captionsetup{width=.925\linewidth}
\caption{Left: $\wh(J)$, where $J$ is the right-handed trefoil. Middle:  a genus-2 surface obtained by Whitehead doubling the genus-1 Seifert surface for $J$. Right: a genus-2 surface obtained by adding a trivial tube to a genus-1 Seifert surface for $\wh(J)$.} 
\label{fig:wh-tref}
\end{figure}

Continuing the analogy with satellite knots, we say that the \emph{wrapping number} of a satellite surface $S$ with respect to a torus $T$ is the number of components in $S \cap T$, and the \emph{winding number} is the absolute value of the signed count of  components in $S \cap T$ as a  collection of oriented curves (i.e., the divisibility of $[S \cap T]$ in $H_1(T)$).

Our final construction  generalizes Example~\ref{ex:WhTref}:

\begin{proposition}\label{prop:notmingenus}
    Let $S \subset S^3$ be a connected, positive-genus Seifert surface for a link $L$, and suppose that $S$ is a satellite  with respect to an incompressible torus in $S^3 \setminus L$. If the wrapping number of $S$ exceeds the winding number, then $L$ has another Seifert surface of genus strictly less than $g(S)$.
\end{proposition}

\begin{proof}Let $T \subset S^3 \setminus L$ be the incompressible torus with respect to which $S$ is a satellite,  and let $r$ and $w$ denote the wrapping and winding numbers of $S$ with respect to $T$, respectively. By definition, $S$ meets one of the components of $S^3 \setminus T$ in $r$ copies of a subsurface $F$. Since $r>w \geq 0$, $F$ must be a nonempty surface with $g(F)\geq 1$.  The torus $T$ intersects $S$ in $r$ parallel curves and, since $r >w$, there must be at least one pair of adjacent curves $C_\pm$ whose orientations (induced by those on $S$ and $T$) are opposite. These curves cobound an annulus $A \subset T$ whose interior is disjoint from $S$, so we may construct a new Seifert surface $S'$ for $L$ by removing the two parallel copies of $F$ bounded by $C_\pm$  and gluing in the annulus $A$. Observe that 
$$g(S')=g(S)-2g(F)+1<g(S),$$
where the $+1$ term in the first equality appears because  $A$ joins two components of a connected surface.
\end{proof}

Proposition \ref{prop:notmingenus} gives rise to a construction of pairs of equal-genus satellite Seifert surfaces for many satellite knots: an initial satellite surface $S$, and a satellite surface of lower wrapping number obtained from the lower-genus surface $S'$ (as constructed in the proof of Proposition~\ref{prop:notmingenus}) by stabilizing it until it has the same genus as $S$. The primary situation that we consider involves Whitehead doubling, as illustrated by Example~\ref{ex:WhTref} and discussed further in \S\ref{subsec:whitehead}.

Before focusing on Whitehead doubles, we pause to consider the effect of torus twisting (as discussed in \S\ref{subsec:spinning}) on satellite surfaces. As demonstrated in  \cite{eisner,kakimizu}, applying torus twists to a satellite Seifert surface can yield infinite families of Seifert surfaces that are distinct up to isotopy rel.~boundary in $S^3$. However, the next proposition shows that such surfaces are smoothly isotopic rel.~boundary in $B^4$.

\begin{proposition}\label{prop:twistisotopy}
Let $S$ be a Seifert surface in $S^3$ that is a satellite with respect to a torus $T$. If $S'$ is obtained from $S$ by meridional twists along $T$, then $S$ and $S'$ are smoothly isotopic rel.~boundary in $B^4$.
\end{proposition}

\begin{proof}
By definition, the torus $T$ bounds a solid torus $V\subset S^3$ that contains the link $L=\partial S$, and $S$ meets $S^3 \setminus \mathring{V}$ in sheets consisting of parallel copies of a fixed subsurface $F$. The surface $S'$ is obtained from $S$  by twisting these sheets  around the $T=\partial V$ as depicted on the left side of Figure~\ref{fig:unroll} (or its mirror through a horizontal plane).

\begin{figure}\center
\def\svgwidth{.95\linewidth}
\begingroup%
  \makeatletter%
  \providecommand\color[2][]{%
    \errmessage{(Inkscape) Color is used for the text in Inkscape, but the package 'color.sty' is not loaded}%
    \renewcommand\color[2][]{}%
  }%
  \providecommand\transparent[1]{%
    \errmessage{(Inkscape) Transparency is used (non-zero) for the text in Inkscape, but the package 'transparent.sty' is not loaded}%
    \renewcommand\transparent[1]{}%
  }%
  \providecommand\rotatebox[2]{#2}%
  \newcommand*\fsize{\dimexpr\f@size pt\relax}%
  \newcommand*\lineheight[1]{\fontsize{\fsize}{#1\fsize}\selectfont}%
  \ifx\svgwidth\undefined%
    \setlength{\unitlength}{673.64060193bp}%
    \ifx\svgscale\undefined%
      \relax%
    \else%
      \setlength{\unitlength}{\unitlength * \real{\svgscale}}%
    \fi%
  \else%
    \setlength{\unitlength}{\svgwidth}%
  \fi%
  \global\let\svgwidth\undefined%
  \global\let\svgscale\undefined%
  \makeatother%
  \begin{picture}(1,0.28344625)%
    \lineheight{1}%
    \setlength\tabcolsep{0pt}%
    \put(0,0){\includegraphics[width=\unitlength,page=1]{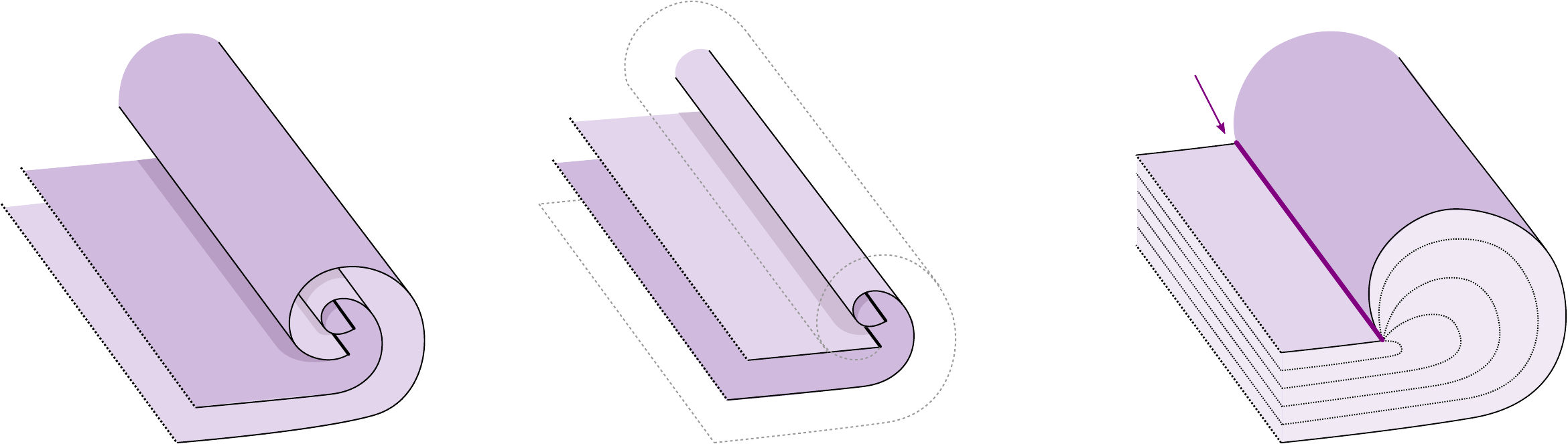}}%
    \put(0.74257074,0.25274113){\color[rgb]{0.50196078,0,0.50196078}\makebox(0,0)[lt]{\lineheight{1.25}\smash{\begin{tabular}[t]{l}$\gamma$\end{tabular}}}}%
    \put(0,0){\includegraphics[width=\unitlength,page=2]{roll_mod.pdf}}%
    \put(0.67966991,0.19902275){\color[rgb]{0.50196078,0,0.50196078}\makebox(0,0)[lt]{\lineheight{1.25}\smash{\begin{tabular}[t]{l}$F_+$\end{tabular}}}}%
    \put(0.94405629,0.21601061){\color[rgb]{0.50196078,0,0.50196078}\makebox(0,0)[lt]{\lineheight{1.25}\smash{\begin{tabular}[t]{l}$F_-$\end{tabular}}}}%
    \put(0,0){\includegraphics[width=\unitlength,page=3]{roll_mod.pdf}}%
  \end{picture}%
\endgroup%

\caption{Left: The surface $S'$, obtained from a satellite surface $S$ by twisting along the satellite torus. Middle: the intermediary surface $\hat{S}$ from the proof of Proposition~\ref{prop:twistisotopy}. Right: the surfaces $F_+$ and $F_-$ cobounding a product in $S^3$.
}\label{fig:unroll}
\end{figure}

As an intermediate step, we consider a surface $\hat{S}$ as depicted in the middle of Figure~\ref{fig:unroll}, obtained from $S'$ by removing a twist from an outermost sheet; note that this has the effect of changing which sheet is on top. Up to isotopy rel.\ boundary, we may assume that $\hat{S}$ and $S'$ coincide everywhere except for the interiors of two subsurfaces $F_- \subset S'$ and $F_+\subset \hat{S}$ as illustrated on the right side of Figure~\ref{fig:unroll}. Moreover, these subsurfaces $F_\pm$ are each isotopic to the original subsurface $F$, and their common boundary is a curve $\gamma$ that is isotopic to $\partial F$.

There is a simple isotopy from $F_-$ to $F_+$ rel.\ $\gamma$ in $S^3$. By pushing the interiors of the intermediate surfaces into $B^4$, we obtain an isotopy from $S'$ to $\hat{S}$ rel.\ boundary in $B^4$. By repeating this argument (such that the total number of iterations is $|S\cap T|$), we construct an isotopy from $S'$ to $S$.
 \end{proof}
 
 \subsubsection{Whitehead doubles.} \label{subsec:whitehead}

 Let $S$ be a Seifert surface for a knot $J$. The {\emph{positive Whitehead double}} $\wh(S)$ of $S$ is a Seifert surface for the positive Whitehead double $\wh(J)$ of $J$ obtained by joining two parallel copies of $S$ (of opposite orientations) via a positively twisted band. Figure~\ref{fig:whitehead_double} depicts  the positive Whitehead double of the surface $\Sigma_1$.

A key motivation for this section is the following theorem of Conway and Powell.

\begin{theorem}[{\cite[Theorem~1.9]{conway-powell}}]\label{thm:conwaypowell}
     Any pushed-in Seifert surfaces of the same genus for a knot with Alexander polynomial 1 are topologically  isotopic rel.~boundary in $B^4$.
\end{theorem}

Returning to the main examples $\Sigma_0$ and $\Sigma_1$ from Figure~\ref{fig:46band}, Theorem~\ref{thm:conwaypowell} implies that $\wh(\Sigma_0)$ and $\wh(\Sigma_1)$ are topologically isotopic rel.\ boundary, even though $\Sigma_0$ and $\Sigma_1$ are not topologically isotopic (as we show in Theorem~\ref{thm:doublecover}). In Section~\ref{section:khovanov}, we use Khovanov homology to obstruct $\wh(\Sigma_0)$ and $\wh(\Sigma_1)$ from being smoothly isotopic.

\begin{figure}
	\centering\includegraphics[width=.85\linewidth]{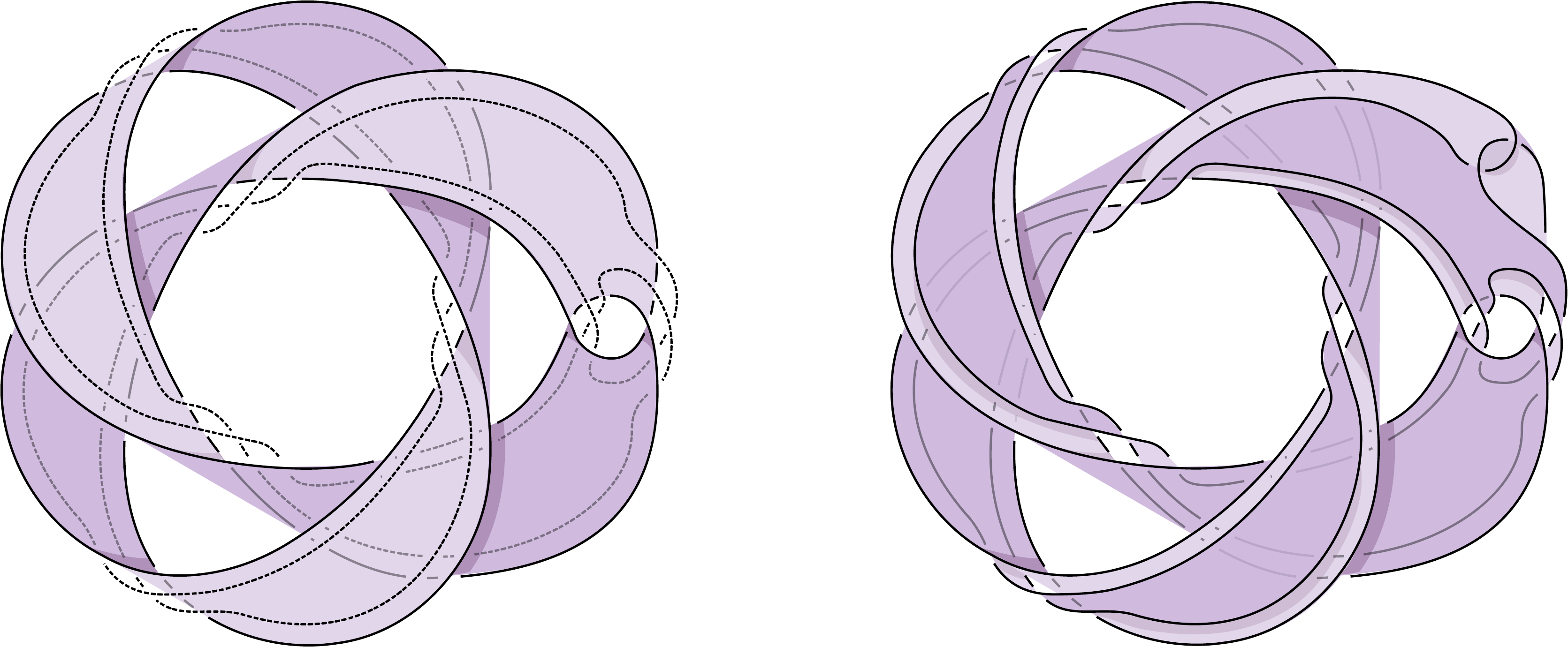}
	\caption{Left: the Seifert pushoff of the knot $K$ and the surface $\Sigma_1$. Right: the positive Whitehead doubled surface  $\wh(\Sigma_1)$.}
	\label{fig:whitehead_double}
\end{figure}

\begin{remark}
    As a technical point, we note that an isotopy between two Seifert surfaces $S_0$ and $S_1$ for a knot $J$ induces an isotopy between the  doubles $\wh(S_0)$ and $\wh(S_1)$. However, the doubled surfaces need not be isotopic rel.\ boundary, even if  $S_0$ and $S_1$ \textsl{are} isotopic rel.\ boundary. Fortunately, for our purposes, Proposition \ref{prop:twistisotopy} ensures that positive Whitehead doubles of surfaces are indeed well-defined up to isotopy rel.\ boundary in $B^4$. Indeed, in order to obtain a positive Whitehead double of a surface $S$, we first have to choose an annulus $A$ between $\partial S$ and the boundary of a tubular neighborhood $N$ of $\partial S$. We isotope $S$ rel.\ boundary to agree with $A$ in $N$ and then double $S$ and attach a band to obtain $\wh(S)$. If we use a different annulus $A'$ that twists one more time about the meridian of $N$, then the resulting surface will differ from the original by a twist along the satellite torus $\partial N$.  (Or in other words, $\wh(S)$ is defined unambiguously outside of a tubular neighborhood of its boundary, up to isotopy fixing the Whitehead satellite torus $T$ setwise.)
\end{remark}

\subsubsection{Extending smooth symmetries} The goal of this subsection is to establish the following proposition, which will play a key role in showing that $\wh(\Sigma_0)$ and $\wh(\Sigma_1)$ are not equivalent under any smooth ambient isotopy (and not merely isotopies that fix the boundary). In the following proposition, recall that $K$ and $\Sigma_1$ are specifically as in Figure~\ref{fig:46band}.

\begin{proposition}\label{prop:extend}
Any diffeomorphism of $S^3$ that fixes $K$ $($resp.~$\wh(K))$ setwise extends to a diffeomorphism of $B^4$ that fixes $\Sigma_1$ $($resp.\  $\wh(\Sigma_1))$ setwise.
\end{proposition}

Given a knot $J \subset S^3$, let $\diff(S^3,J)$ denote the group of diffeomorphisms of $S^3$ that fix $J$ setwise. The symmetry group of $J$, denoted $\sym(J)$, is the quotient of the group $\diff(S^3,J)$  by $\diff_0(S^3,J)$, the normal subgroup of diffeomorphisms that are isotopic to the identity through diffeomorphisms of the pair $(S^3,J)$.  The  following lemma reduces Proposition~\ref{prop:extend} to an analysis of the discrete, finitely generated group $\sym(J)$ instead of the infinite-dimensional group $\diff(S^3,J)$.

\begin{lemma}
Let $S \subset B^4$ be a smooth, properly embedded surface bounded by a knot $J \subset S^3$. If every element of $\sym(J)$ has a representative in $\diff(S^3,J)$ that extends to a diffeomorphism of $(B^4,S)$, then every diffeomorphism of $(S^3,J)$ extends to a diffeomorphism of $(B^4,S)$.
\end{lemma}

\begin{proof}
Let $f$ be any diffeomorphism of the pair  $(S^3,J)$. By hypothesis, there exists $g \in \diff(S^3,J)$ representing the same element of $\sym(J)$ such that $g$ extends to a diffeomorphism of $(B^4,S)$. To show that $f$ extends to a diffeomorphism of $(B^4,S)$, it then suffices to show that  $g^{-1} \circ f$ extends, for we may then compose  with the extension of $g$ to obtain an extension of $f$. 

The equivalence of $f$ and $g$ in $\sym(J)$ implies that these maps are isotopic through diffeomorphisms of  $(S^3,J)$. This implies that $g^{-1}\circ f$ is isotopic to the identity through diffeomorphisms of $(S^3,J)$. Thus, it suffices to prove that every element of $\diff_0(S^3,J)$ extends to a diffeomorphism of $(B^4,S)$.

For notational convenience, view $B^4 \setminus \{pt\}$ as $S^3 \times (-\infty,1]$ and  choose an isotopy $\psi_t:B^4 \to B^4$ (rel.\ $\partial B^4$) that straightens a collar neighborhood of $S$, i.e.\ where $\psi_0=\id$ and $\psi_1(S) \cap  ( S^3 \! \times \! (0,1])= J \times (0,1]$. Given an element $h\in \diff_0(S^3,J)$, choose an isotopy $h_t$ with $h_0=\id$ and $h_1=h$. Now define $H_t:B^4\to B^4$ by $H_t(x,r)=(h_{rt}(x),r)$ for all $(x,r) \in S^3 \times (0,1]$ and by the identity on the rest of $B^4$. Then $\psi_t^{-1} \circ H_t \circ \psi_t$ is an isotopy from the identity at $t=0$ to a new   diffeomorphism  at $t=1$ that extends $h$ to all of $B^4$ and maps $S$ to itself, as desired.
\end{proof}

\begin{proof}[Proof of Proposition~\ref{prop:extend}] 
We will write the proof just for $\wh(\Sigma_1)$; the proof for $\Sigma_1$ is implicit and strictly easier. The knot $K$ is itself the $-6$-framed positive Whitehead double of the left-handed trefoil. Thus, the complement of $\wh(K)$ admits a JSJ decomposition consisting of the following pieces:
\begin{enumerate}
    \item $M_1$, a left-handed trefoil complement,
    \item $M_2$, the complement of the $-6$-framed positive Whitehead link,
    \item $M_3$, the complement of the standard Whitehead link.
\end{enumerate}

Any self-homeomorphism $f$ of $(S^3, \wh(K))$ must fix this decomposition up to isotopy. The map $f$ is determined up to isotopy by an automorphism of each $M_i$ that fixes the ordering of its boundary components, up to twists about the tori boundary of each $M_i$.

Observe that $M_2$ and $M_3$ are homeomorphic as 3-manifolds (although not by a homeomorphism preserving meridians of the associated links). Both  these manifolds  have automorphism group $D_4$ (see \cite[Table~2]{henry-weeks}), but this includes a map that switches the two boundary components. Quotienting by this map, the automorphism groups of $M_2$ and $M_3$ fixing each boundary component setwise are $\mathbb{Z}_2\oplus\mathbb{Z}_2$.

\begin{figure}
\labellist
\pinlabel $\Sigma_1\cap M_2$ at 62 -10
\pinlabel $\Sigma_1\cap M_3$ at 189 -10
\endlabellist
    \includegraphics[width=.6\linewidth]{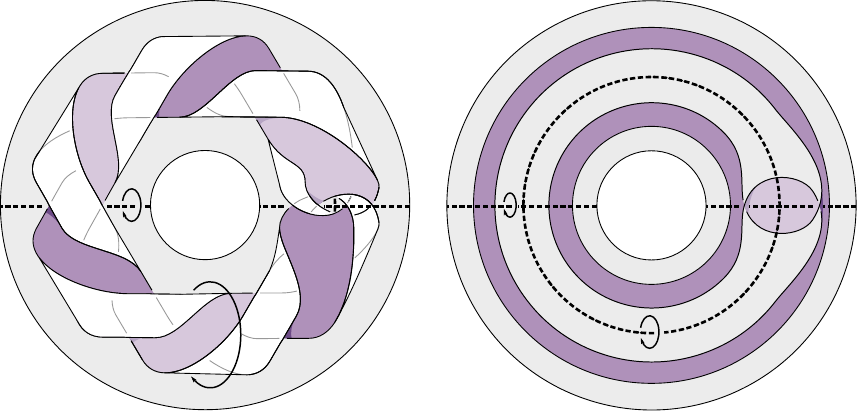}
    \vspace{.15in}
    
    \caption{Intersections of $\wh(\Sigma_1)$ with the JSJ pieces $M_2$ and $M_3$. The manifold $M_2$ is the complement of the $-6$-framed positive Whitehead link; $\Sigma_1$ intersects $M_2$ in two components that twist about the ``inner" boundary. The manifold $M_3$ is the complement of the usual Whitehead link; here one boundary is filled 
    to indicate $\wh(K)$.}
    \label{fig:symmetrygroup}
\end{figure}

We temporarily project $\wh(\Sigma_1)$ back into $S^3=\partial B^4$ and consider its intersections with $M_1, M_2$, and $M_3$. Note that $\wh(\Sigma_1)$ does not intersect $M_1$.  We illustrate the intersections of $\wh(\Sigma_1)$ with $M_2$ and $M_3$ in Figure~\ref{fig:symmetrygroup}. The automorphism group of $M_2$ fixing boundary components is generated in Figure~\ref{fig:symmetrygroup} (left) by 180$^{\circ}$ rotation about a horizontal axis and 180$^{\circ}$ rotation about a circular axis in the plane of the drawing. Both of these maps fix $\wh(\Sigma_1)\cap M_2$ setwise. The part of the automorphism group of $M_3$ preserving both boundary components is similarly generated in Figure~\ref{fig:symmetrygroup} (right) by rotation about a horizontal axis and rotation about a circular axis in the plane of the drawing. Again, both of these maps fix $\wh(\Sigma_1)\cap M_3$ setwise.

So far, we have seen that automorphisms of $M_2$ and $M_3$ are isotopic to ones that fix the portion of $\wh(\Sigma_1)$ inside $M_2$ or $M_3$ setwise (even in $S^3$). Thus, we need only consider twists about the JSJ tori. Moreover, since $\wh(\Sigma_1)$ does not intersect the torus separating $M_1$ and $M_2$, we need only consider the tori that cobound $M_3$. We now push the interior of $\wh(\Sigma_1)$ slightly back into $B^4$. One of these tori is the boundary of a tubular neighborhood of $\wh(K)$; twists about this torus can be extended over this solid torus taking $\wh(K)$ to be a fixed core. The other torus is the Whitehead satellite torus. By Propostion~\ref{prop:twistisotopy}, twists about this torus preserves  $\wh(\Sigma_1)$ up to isotopy rel.\ boundary in $B^4$.
\end{proof}

\section{Obstructions from Khovanov homology} \label{section:khovanov}

\subsection{Khovanov preliminaries}\label{sec:preliminaries}

Khovanov homology is known to be functorial for link cobordisms in $\R^3 \times [0,1]$, in the sense that a smooth, oriented, properly embedded cobordism $S \subset \R^3 \times [0,1]$ between links $L_0$ and $L_1$ induces a diagramatically-defined $\Z_2$-linear map on Khovanov homology $\kh(S): \kh(L_0) \to \kh(L_1)$ that is invariant under smooth isotopy of $S$ rel.\ boundary \cite{khovanov00, jacobsson, barnatan, khovanov06}. We make note of an important extension of this invariant from \cite[Lemma 4.7]{morrison-walker-wedrich} and \cite[Proposition 3.7]{lipshitz-sarkar}, where it is proven that a (possibly non-orientable) link cobordism in $S^3 \times [0,1]$ induces a map on Khovanov homology that is invariant under diffeomorphism of $S^3 \times [0,1]$ that restricts to the identity on the boundary.

The Khovanov functor can be adapted to the setting of surfaces in $B^4$. To that end, suppose $S_0, S_1 \subset B^4$ are surfaces bounded by the same link $L \subset S^3$. Any boundary-preserving diffeomorphism of $B^4$ carrying $S_0$ to $S_1$ can be taken to fix a small open ball in the complement of the surfaces (via an isotopy supported far from the surfaces). This induces a diffeomorphism of link cobordisms in $S^3 \times [0,1]$. Thus, to distinguish the  surfaces $S_0$ and $S_1$ in $B^4$, it suffices to distinguish the link cobordisms $L \to \emptyset$ they induce in $S^3 \times [0,1]$, or equivalently, to distinguish the maps $\kh(L) \to \Z_2$ they induce on Khovanov homology.

Previous work has proven the efficacy of this technique by showing that the maps on Khovanov homology  distinguish many families of surfaces in $B^4$ up to diffeomorphism rel.\ boundary, including exotic examples \cite{hayden-sundberg, lipshitz-sarkar, sundberg-swann}. Moreover, the approach in \cite{hayden-sundberg} demonstrates the computability of such maps: by carefully choosing a cycle $\psi \in \ckh(L)$ from the Khovanov chain complex of $L$, we can control the complexity of calculating the induced chain maps $\ckh(S_0)(\psi)$ and $\ckh(S_1)(\psi)$.  Moreover, the calculations in this paper only require a subset of the Morse and Reidemeister induced chain maps (c.f.\  \cite[Tables 1-2]{hayden-sundberg}), tailored to $x$-labeled smoothings with $\Z_2$ coefficients.

\subsection{Main example computation} \label{subsection:khovanov_main}

To motivate the main obstruction from Khovanov homology in Theorem~\ref{thm:exotic}, we begin by distinguishing the maps on Khovanov homology induced by $\Sigma_0$ and $\Sigma_1$. In combination with Proposition~\ref{prop:extend} (c.f.\@~Corollary~\ref{cor:ambient}), this proves Theorem~\ref{thm:ambient} in the smooth setting. Moreover, the computations of $\ckh(\Sigma_0)$ and $\ckh(\Sigma_1)$ set the groundwork from which we distinguish the maps induced by $\wh(\Sigma_0)$ and $\wh(\Sigma_1)$.

\begin{proposition}\label{prop:khovanov}
    The Seifert surfaces $\Sigma_0$ and $\Sigma_1$ induce distinct maps on Khovanov homology, distinguished by a given cycle $\phi \in \ckh(K)$, and hence are not related by any diffeomorphism of $B^4$ that restricts to the identity on $\partial B^4$ $($and in particular are not smoothly isotopic rel.\ boundary in $B^4)$.
    \end{proposition}
\begin{proof}
    A smoothing of $K$ is given on the top-right of Figure~\ref{fig:phi_calculation}, and a straightforward calculation shows that $x$-labeling each component produces a cycle $\phi \in \ckh(K)$ in the Khovanov chain complex ($0$-smoothed crossings are indicated by a gray band on $\phi$ and connect distinct $x$-labeled components). We immediately note that the map on the Khovanov chain complex induced by $\Sigma_0$ satisfies $\ckh(\Sigma_0)(\phi) = 0$, as one can find many band moves describing $\Sigma_0$ whose induced map merges distinct $x$-labeled components of $\phi$. A complementary calculation for the chain map induced by $\Sigma_1$ is given in Figure~\ref{fig:phi_calculation} (completed after capping off the remaining unknots), and shows  $\ckh(\Sigma_1)(\phi) = 1$. We conclude that $\phi$ distinguishes the maps induced by these surfaces, implying that there is no diffeomorphism of $B^4$ restricting to the identity on $\partial B^4$ that sends $\Sigma_0$ to $\Sigma_1$.
\end{proof}

\begin{figure}[p]
	\centering
	\includegraphics[width=.875\linewidth]{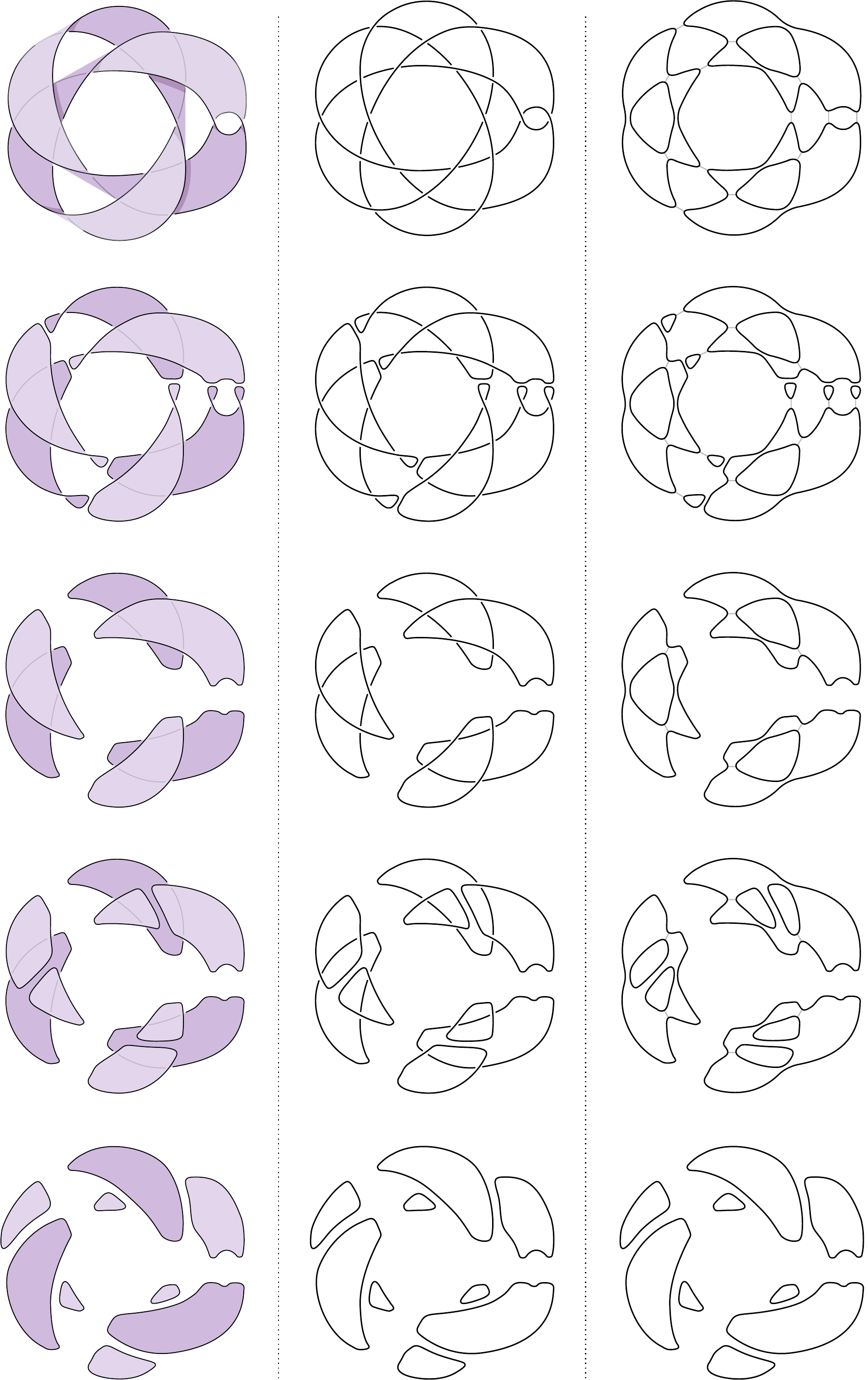}
 \captionsetup{width=.85\linewidth}
	\caption{Left: the surface $\Sigma_1$ as it is dipped into $B^4$. Middle: the corresponding movie. Right: the effect of each move on $\phi$, with all smoothings $x$-labeled.}
	\label{fig:phi_calculation}
\end{figure}

\begin{remark} \label{remark:localization}
    The movie of $\Sigma_1$ that we chose in Figure~\ref{fig:phi_calculation} localizes the induced chain map on Khovanov homology. In particular, $\Sigma_1$ can be decomposed into a collection of \textit{band twists} \raisebox{-.2\height}{\includegraphics[scale=.5]{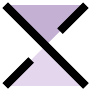}} and \textit{band crossings} \raisebox{-.2\height}{\includegraphics[scale=.5]{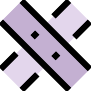}}, with which we compute $\ckh(\Sigma_1)$ as a collection of isolated chain maps on the corresponding boundary tangles (c.f.\ \cite{khovanov02,barnatan}). The chosen tangle decomposition of $K$ sheds light on how we chose a cycle: within each tangle, $\phi$ has a consistent labeled smoothing, where band twists have oriented smoothings, band crossings have nonoriented smoothings, and only $x$ labels are used throughout. In the next section, we show that this localization is, in some sense, stable under the process of Whitehead doubling from Section~\ref{subsec:satellite}.
\end{remark}

\subsection{Whitehead doubling and Khovanov homology} \label{subsection:khovanov_whitehead}

We now distinguish the maps on Khovanov homology induced by $\wh(\Sigma_0)$ and $\wh(\Sigma_1)$. Motivated by the previous section: we decompose $\wh(K)$ into a collection of tangles which reflect the local behavior of $\wh(\Sigma_1)$, we choose a cycle $\Phi \in \ckh(\wh(K))$ by choosing labeled smoothings for each tangle, and finally, we calculate the induced maps on $\Phi$ as a collection of tangle maps.

\begin{theorem}\label{thm:khovanov}
The surfaces $\wh(\Sigma_0)$ and $\wh(\Sigma_1)$ induce distinct maps on Khovanov homology, distinguished by a given cycle $\Phi \in \ckh(\wh(K))$, and hence, are not related by a diffeomorphism of $B^4$ restricting to the identity on $\partial B^4$ $($and in particular are not smoothly isotopic rel.\ boundary in $B^4)$.
\end{theorem}

Theorem~\ref{thm:khovanov} implies Theorem~\ref{thm:exotic} in the finite case and up to smooth equivalence rel.\ boundary (see the discussion in Section~\ref{subsec:whitehead} and, in particular, Theorem~\ref{thm:conwaypowell}).

\begin{figure}[p]
    {\def\svgwidth{.88\linewidth}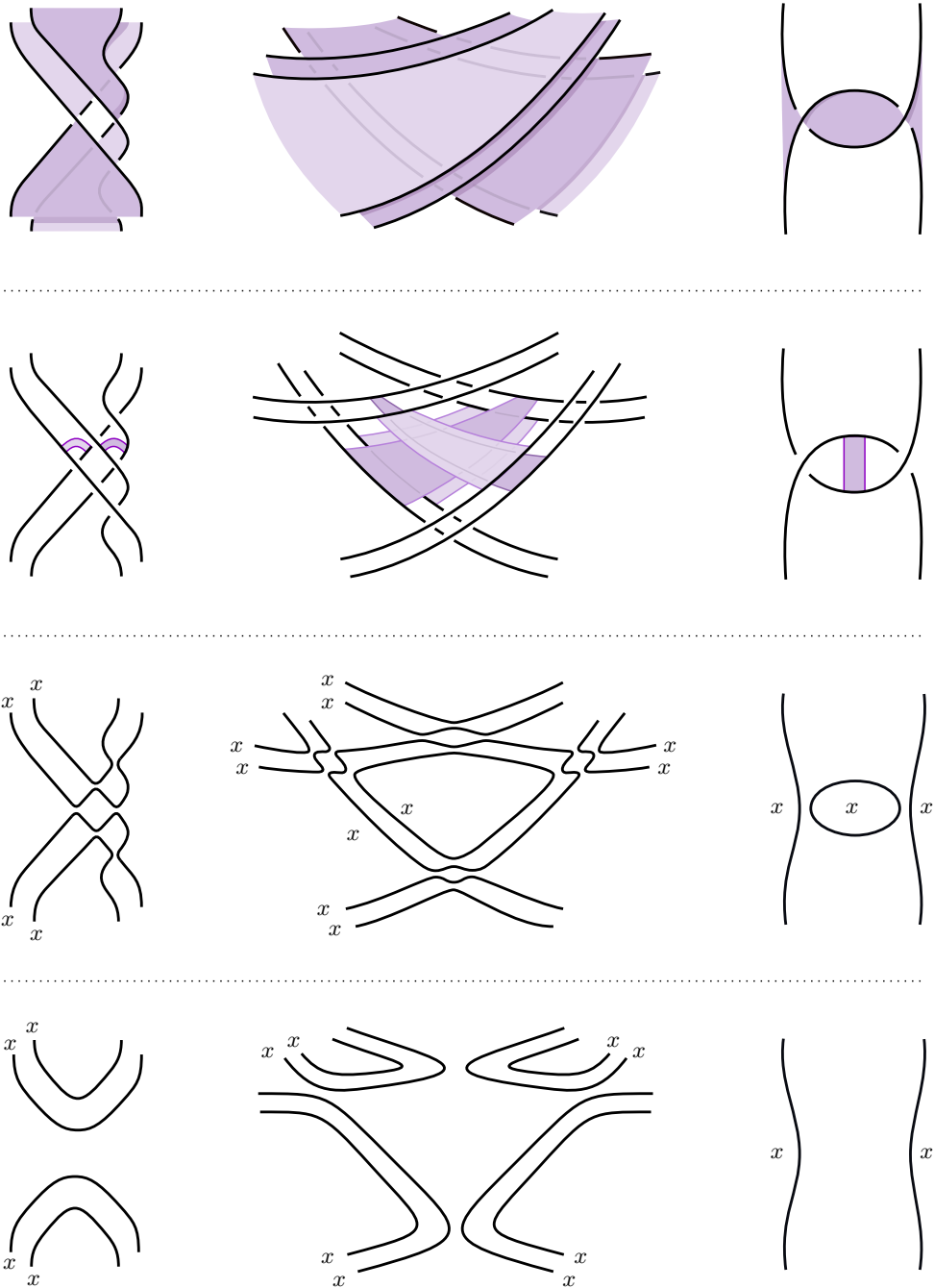}
    
    \vspace{.05in}
    
    \caption{By row, top: local pictures of $\wh(\Sigma_1)$; boundary tangles with bands reflecting the surface they bound; a chosen labeled smoothing for each tangle; the result of applying the maps induced by the band moves (and a sequence of Reidemeister moves) to each labeled smoothing.}
    \label{fig:bigs}
\end{figure}

\begin{proof}[Proof of Theorem~\ref{thm:khovanov}]
    To begin, observe that $\wh(\Sigma_1)$ exhibits three types of local behavior, illustrated in the first row of Figure~\ref{fig:bigs}: \textit{big twists}, \textit{big crossings}, and a single \textit{Whitehead clasp}. The boundary of each is illustrated in the second row as a tangle in $\wh(K)$, decorated with band moves to reflect the local behavior of the surface $\wh(\Sigma_1)$. In the third row, we choose a labeled smoothing for each tangle. We extend to a labeled smoothing $\Phi \in \ckh(\wh(K))$ by noting that the strands connecting any two tangles are consistently $x$-labeled.
    
    We may dip $\wh(\Sigma_1)$ into $S^3 \times [0,1]$ to produce a movie consisting of two stages: first, the band moves within each tangle, followed by a sequence of Reidemeister moves that simplify each tangle; second, a sequence of Morse deaths that cap off the resulting crossingless unlink. The first stage is localized within each tangle in Figure~\ref{fig:bigs}, so the map it induces is localized to labeled smoothings of that tangle. When applied to the relevant labeled smoothings for a big twist, big crossing, and Whitehead clasp, we obtain the final row of Figure~\ref{fig:bigs}. Overall, the first stage maps $\Phi$ to the all $x$-labeled smoothing for the crossingless unlink. The final stage sends this labeled smoothing to $1$, so collectively, $\ckh(\wh(\Sigma_1))(\Phi) = 1$.

    Now consider $\wh(\Sigma_0)$ and, in particular, the local behavior near a big crossing for $\wh(\Sigma_1)$, as illustrated in Figure~\ref{fig:bad_band}. The highlighted band move can be realized as the composition of a Reidemeister II move and a saddle. The induced map  on the labeled smoothing for a big crossing leads to merge maps between distinct $x$-labeled components, from which it follows that $\ckh(\wh(\Sigma_0))(\Phi) = 0$.
\end{proof}
    
\begin{figure}
    \includegraphics[width=.75\linewidth]{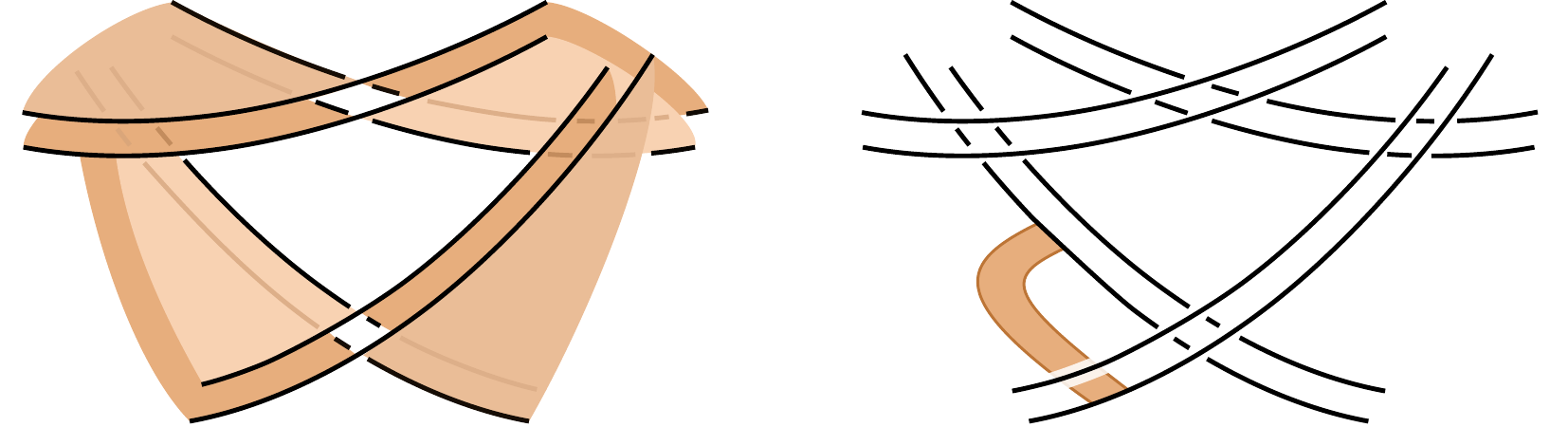}
    \caption{A local picture of $\wh(\Sigma_0)$ near a big crossing for $\wh(\Sigma_1)$, together with a band used to calculate $\ckh(\wh(\Sigma_0))$.}
    \label{fig:bad_band}
\end{figure}

 \begin{remark}\label{remark:phi_prime}
    Note that the process for choosing $\Phi$ depended entirely on the decomposition of $\wh(K)$ into tangles that reflect the overall surface $\wh(\Sigma_1)$. An entirely identical argument can be used to produce a cycle $\Phi'$ that reflects $\wh(\Sigma_0)$, reversing the roles of the induced maps in the calculations. In particular, we note both maps induced by $\wh(\Sigma_0)$ and $\wh(\Sigma_1)$ are nontrivial.
\end{remark}

\begin{remark}
    Although the surfaces $\wh(\Sigma_0)$ and $\wh(\Sigma_1)$ are not minimal genus in $B^4$, they are not destabilizable. This is because any surface that is a connected sum with a standard $T^2$ will induce the trivial map on Khovanov homology with $\Z_2$ coefficients (or more generally for $\Z$ coefficients, will have image inside $2\Z \subset \Z$), whereas the maps induced by these surfaces are both nontrivial (c.f.\ ~\ref{thm:khovanov} and~\ref{remark:phi_prime}).
\end{remark}

\begin{remark}\label{rem:nonorientgenus}
    Non-orientable and higher-genus examples can be obtained by slight modifications to our base example. Replace the clasp on $K$ (from Whitehead doubling a trefoil) with an $-n + \frac1{-m}$ rational tangle, illustrated in Figure~\ref{fig:tangle}, and produce surfaces analogous to $\Sigma_0$ and $\Sigma_1$. Note that $m=n=1$ corresponds to $K$. Increasing $n$ increases the genus of the surface; increasing $m$ gives an infinite family for each genus (both orientable and non-orientable). As these modifications consist entirely of left-handed crossings, the above arguments persist, so Theorems~\ref{thm:ambient} and~\ref{thm:exotic} hold for non-orientable, higher-genus surfaces in the smooth rel.\ boundary setting.
\end{remark}

\begin{figure}
    \centering
    {\def\svgwidth{.5\linewidth}
\begingroup%
  \makeatletter%
  \providecommand\color[2][]{%
    \errmessage{(Inkscape) Color is used for the text in Inkscape, but the package 'color.sty' is not loaded}%
    \renewcommand\color[2][]{}%
  }%
  \providecommand\transparent[1]{%
    \errmessage{(Inkscape) Transparency is used (non-zero) for the text in Inkscape, but the package 'transparent.sty' is not loaded}%
    \renewcommand\transparent[1]{}%
  }%
  \providecommand\rotatebox[2]{#2}%
  \newcommand*\fsize{\dimexpr\f@size pt\relax}%
  \newcommand*\lineheight[1]{\fontsize{\fsize}{#1\fsize}\selectfont}%
  \ifx\svgwidth\undefined%
    \setlength{\unitlength}{400.82446001bp}%
    \ifx\svgscale\undefined%
      \relax%
    \else%
      \setlength{\unitlength}{\unitlength * \real{\svgscale}}%
    \fi%
  \else%
    \setlength{\unitlength}{\svgwidth}%
  \fi%
  \global\let\svgwidth\undefined%
  \global\let\svgscale\undefined%
  \makeatother%
  \begin{picture}(1,0.38527797)%
    \lineheight{1}%
    \setlength\tabcolsep{0pt}%
    \put(0.63211648,0.17854414){\color[rgb]{0,0,0}\makebox(0,0)[lt]{\lineheight{1.25}\smash{\begin{tabular}[t]{l}{{$m$}}\end{tabular}}}}%
    \put(0.87162282,0.242163){\color[rgb]{0,0,0}\makebox(0,0)[lt]{\lineheight{1.25}\smash{\begin{tabular}[t]{l}{{$n$}}\end{tabular}}}}%
    \put(0,0){\includegraphics[width=\unitlength,page=1]{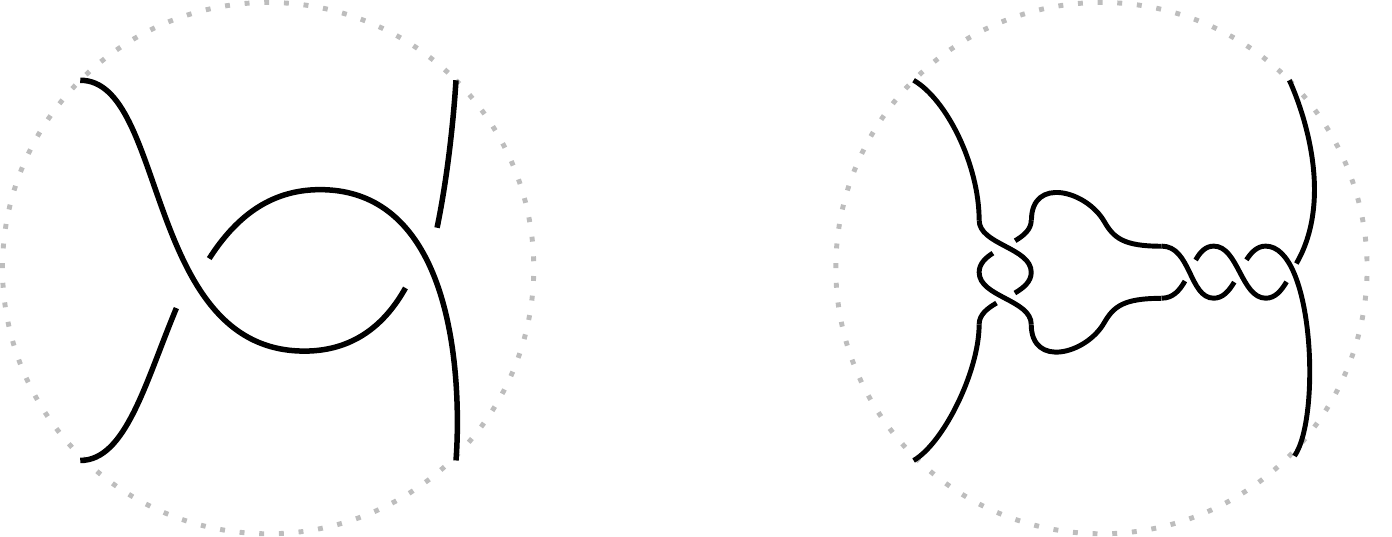}}%
  \end{picture}%
\endgroup%
}
    \caption{Replacing the clasp from $K$ (left) with an $-n + \frac1{-m}$  rational tangle (right) produces an infinite family of knots bounding distinct Seifert surfaces for each positive (non-)orientable genus. (Note that $m, n$ denote numbers of half-twists rather than full twists.)}
    \label{fig:tangle}
\end{figure}

\begin{remark}
    The techniques we used to prove Theorem~\ref{thm:khovanov} appear robust enough to generalize to calculations on larger families of Seifert surfaces. In particular, we expect that whenever a pair of Seifert surfaces for a strongly quasipositive knot $J$ induce distinct maps on Khovanov homology, then so will their positive Whitehead doubles. Moreover, the distinction of their induced maps will be captured by a readily available cycle from $\ckh(\wh(J))$,  produced in a similar manner to Theorem~\ref{thm:khovanov}.
\end{remark}

\begin{corollary}\label{cor:ambient}
The surfaces $\wh(\Sigma_{0})$ and $\wh(\Sigma_1)$ are not related by any diffeomorphism of $B^4$  $($and in particular are not ambiently smoothly isotopic$)$.
\end{corollary}

\begin{proof}
Suppose there exists a diffeomorphism $f \colon (B^4,\wh(\Sigma_0)) \to (B^4,\wh(\Sigma_1))$. By Proposition~\ref{prop:extend}, the diffeomorphism  $f^{-1}{\big|_{S^{3}}}\colon (S^3,\wh(K)) \to (S^3,\wh(K))$ extends to a diffeomorphism $g\colon (B^4,\wh(\Sigma_1))\to (B^4,\wh(\Sigma_1))$. The composition $g\circ f:(B^4,\wh(\Sigma_0))\to (B^4,\wh(\Sigma_1))$ is the identity on $S^3$, contradicting Theorem~\ref{thm:khovanov}.
\end{proof}

\subsection{Whitehead doubles of strongly quasipositive knots}\label{subsection:sqp}

In \cite{rudolph:kauffman-bound}, Rudolph introduced the class of \emph{strongly quasipositive} knots. Though originally defined braid-theoretically, it is convenient (and equivalent) to define such knots as the boundaries of \emph{strongly quasipositive Seifert surfaces}, which are formed from a stack of parallel disks  that are joined by embedded, positively half-twisted bands as in Figure~\ref{fig:sqp}(a). (These disks and bands are oriented so that the surface's boundary is naturally braided.)

\begin{proposition}\label{prop:sqp}
If $S$ is a strongly quasipositive Seifert surface for a knot $J$, then $\wh(S)$ induces a nontrivial map from $\kh(\wh(J))$ to $\mathbb{Z}_2$.
\end{proposition}

\begin{proof}
Given $S$ formed from $n$ disks joined by $\ell$ bands as in Figure~\ref{fig:sqp}(a), we first form the normal pushoff of $S$ by introducing an additional $n$ disks and $\ell$ bands as in Figure~\ref{fig:sqp}(b). We give this pushoff the opposite orientation as that of $S$. To complete $\wh(S)$, we add one more band between the top two sheets as in Figure~\ref{fig:sqp}(c).

We consider an element in $\kh(\wh(J))$ that is essentially the Whitehead double of Plamenevskaya's invariant $\psi(J)$ (when $J$ is viewed as an $n$-stranded braid) \cite{plamenevskaya:transverse-Kh}. Recall that Plamenevskaya's chain element is given by the smoothing of the $n$-braid $J$ as $n$ concentric circles as in Figure~\ref{fig:sqp}(d), with all circles labeled with an $x$. Consider the similarly constructed element of $\ckh(\wh(J))$, where the regions from parts (b) and (c) of Figure~\ref{fig:sqp} are smoothed and labeled as in parts (e) and (f). All 0-resolved crossings join distinct $x$-labeled circles, so this chain element is a cycle. A calculation entirely analogous to the one given in Figure~\ref{fig:bigs} shows that $\kh(\wh(S))$ maps the described element of $\kh(\wh(J))$ to 1 in $\mathbb{Z}_2=\kh(\emptyset)$, as desired.
\end{proof}

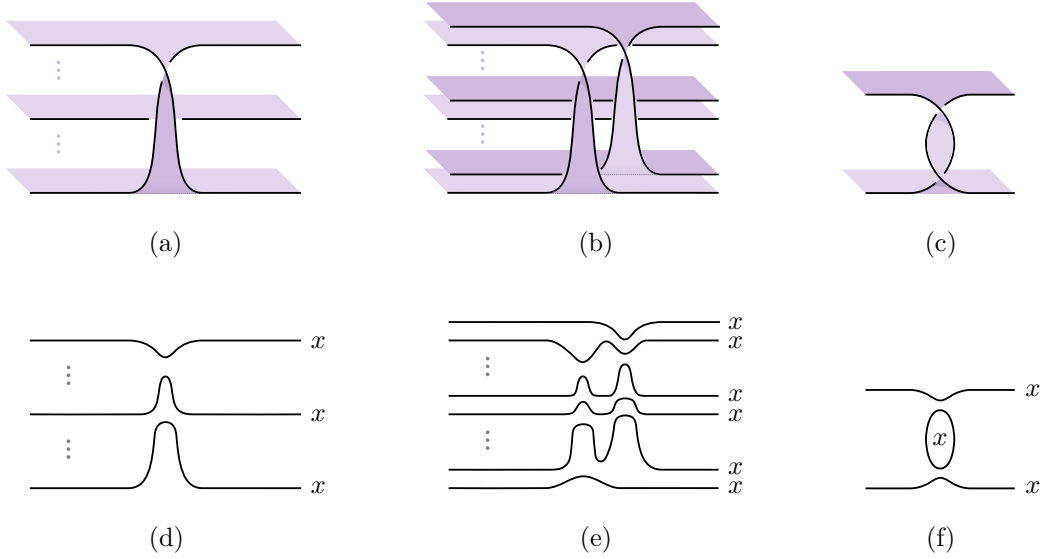
\begin{figure}
       \labellist
\small\hair 2pt
\pinlabel (a) at 49 75
\pinlabel (b) at 180 75
\pinlabel (c) at 285 75
\pinlabel (d) at 49.25 -15
\pinlabel (e) at 180 -15
\pinlabel (f) at 285 -15

\pinlabel $x$ at 95.5 45
\pinlabel $x$ at 95.5 23
\pinlabel $x$ at 95.5 0.75

\pinlabel $x$ at 222.5 51
\pinlabel $x$ at 222.5 45
\pinlabel $x$ at 222.5 29
\pinlabel $x$ at 222.5 23
\pinlabel $x$ at 222.5 6.75
\pinlabel $x$ at 222.5 0.75

\pinlabel $x$ at 284.75 15.5
\pinlabel $x$ at 313 0.75
\pinlabel $x$ at 313 30.5
\endlabellist
\includegraphics[width=.925\linewidth]{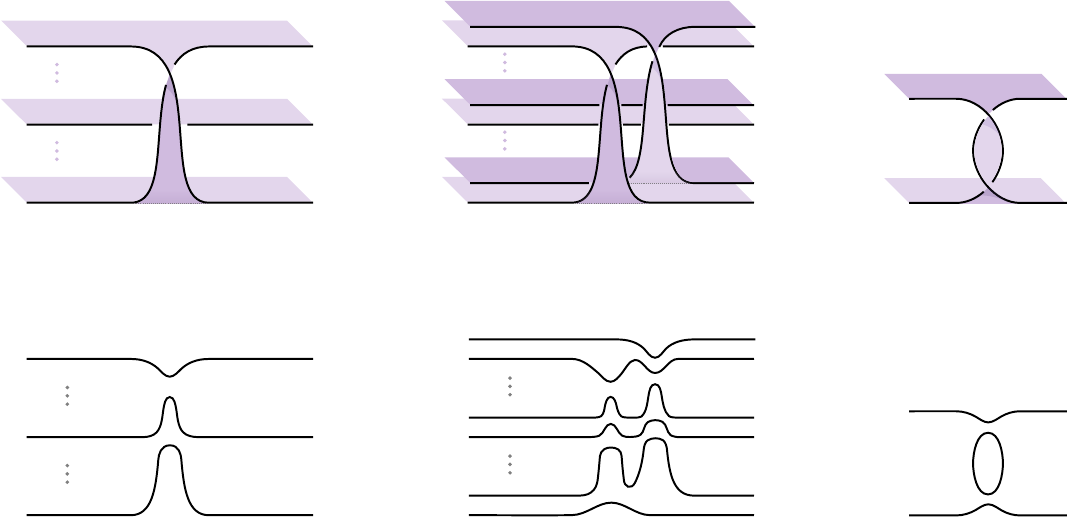}
\bigskip
\bigskip
\captionsetup{width=.85\linewidth}
\caption{(a) A half-twisted band in a strongly quasipositive Seifert surface. (b) Whitehead doubling near a half-twisted band. (c) A fully twisted band forming the clasp in the Whitehead double. (d-e) Labeled smoothings of these regions.}\label{fig:sqp}
\end{figure}

We may now prove Theorem~\ref{thm:sqp}, which we restate slightly more precisely.

\begin{theorem}
If $J$ is a nontrivial strongly quasipositive knot of Seifert genus $g$, then $\wh(J)$ bounds at least two Seifert surfaces of genus $2g$ that are topologically isotopic rel.~boundary in $B^4$ yet are not related by any diffeomorphism of $B^4$ (and in particular are not ambiently smoothly isotopic).
\end{theorem}

\begin{proof}
Let $S$ be a strongly quasipositive Seifert surface for $J$, which has genus $g\geq 1$ because $J$ is not the unknot. Let $\wh(S)$ denote its Whitehead double, which has genus $g(\wh(S))=2g\geq 2$. By the proposition above, $\wh(S)$ induces a nontrivial map on Khovanov homology with $\mathbb{Z}_2$-coefficients.

On the other hand, consider the genus-$2g$ Seifert surface for $\wh(J)$ obtained from the standard genus-1 Seifert surface for $\wh(J)$ by stabilizing $2g-1$ times. Any stabilized surface induces a trivial map on Khovanov homology with $\mathbb{Z}_2$-coefficients. Since the Whitehead doubled surface $\wh(S)$ induces a nontrivial map, it cannot be destabilized, hence there can be no diffeomorphism of $B^4$ carrying $\wh(S)$ to the stabilized surface.

On the other hand, these genus-$2g$ Seifert surfaces are topologically isotopic rel.\ boundary in $B^4$ by Theorem~\ref{thm:conwaypowell}.
\end{proof}

\subsection{Minimal genus examples and discussion}\label{subsec:furtherexamples}

We now modify our previous examples to obtain the minimal genus examples claimed in  Theorem~\ref{thm:exotic}.   Note that the  (untwisted) Whitehead double of any knot bounds a standard genus-1 Seifert surface, so the  Seifert surfaces constructed in Section~\ref{subsec:satellite} will never minimize the Seifert genus of their boundary (see Proposition~\ref{prop:notmingenus}).

To this end, we consider a knot $K_T$ obtained as a band sum of $\wh(K)$ with the right-handed trefoil as shown in  Figure~\ref{fig:band_sum}. Note that we have not specified how many full twists are in the band, so $K_T$ actually refers to any member of an infinite family of knots. These knots can be distinguished by their Khovanov homologies \cite{josh} or by their hyperbolic volumes; see the ancillary files \cite{ancillary}.

Because the band has been chosen to avoid both $\wh(\Sigma_0)$ and $\wh(\Sigma_1)$, we can sum these surfaces with the standard Seifert surface for the trefoil $T$ to produce two genus-3 Seifert surfaces $\Sigma_0^T$ and $\Sigma_1^T$ for $K_{T}$. On the other hand, the band intersects the standard genus-1 Seifert surface for $\wh(K)$ in an essential way, leading to an increase in the Seifert genus.  Moreover, this eliminates all symmetries of the knot, which allows us to avoid an analysis like the one given in Proposition~\ref{prop:extend}.

\begin{figure}[b]
    \centering
    \includegraphics[width=.6\linewidth]{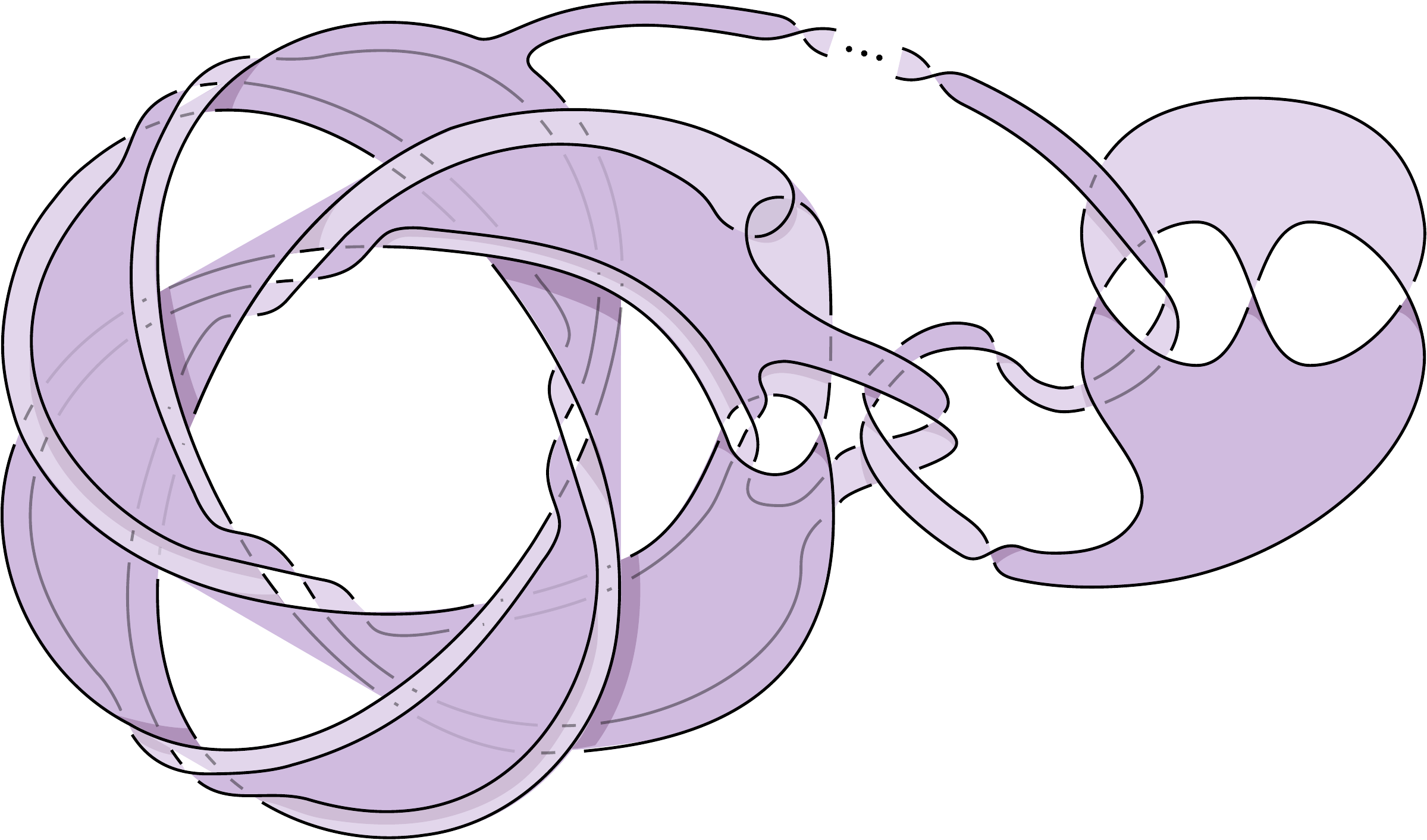}    \caption{The Seifert surface $\Sigma_1^T$ obtained by band summing $\wh(\Sigma_1)$ and a fiber for the trefoil (where the ellipsis indicates any number of full twists). The boundary  of the surface is the knot $K_T$.} 
    
    \label{fig:band_sum}
\end{figure}

\begin{proposition}\label{prop:ktgenus}
The knot $K_{T}$ has Seifert genus $3$ and trivial symmetry group.
\end{proposition}

\begin{proof}
The Seifert genus of $K_{T}$ can be computed with the HFKCalculator \cite{hfk-calc}; the symmetry group of $S^3\setminus K_{T}$ (along with hyperbolicity of this complement) can be checked in SnapPy \cite{snappy} inside Sage \cite{sagemath}. We refer the reader to the ancillary files of this paper \cite{ancillary} for documentation.
\end{proof}

\begin{proof}[Proof of Theorem~\ref{thm:exotic}]
By Proposition~\ref{prop:ktgenus}, we have that $\Sigma_0^T$ and $\Sigma_1^T$ are minimal genus Seifert surfaces. Since $\wh(\Sigma_0)$ and $\wh(\Sigma_1)$ are topologically isotopic rel.\ boundary when pushed into $B^4$, so are $\Sigma_0^T$ and $\Sigma_1^T$.

The band moves shown on the left side of Figure~\ref{fig:band_sum_phi} determine a link cobordism $S$ from $K_T$ to $\wh(K)$, and we may view the surfaces $\Sigma_0^T$ and $\Sigma_1^T$ as compositions $ \wh(\Sigma_0) \circ S$ and $ \wh(\Sigma_1) \circ S$. Next, let $\Psi \in \ckh(K_T)$ be the labeled smoothing that  extends $\Phi \in \ckh(\wh(K))$ using the right half of Figure~\ref{fig:band_sum_phi}. It is straightforward to check that $\Psi$ is a cycle and that the map $\ckh(S)$ sends $\Psi$ to $\Phi$.  It follows that  $\ckh(\Sigma_0^T)(\Psi) = \ckh(\wh(\Sigma_0))(\Phi) = 0$ and $\ckh(\Sigma_1^T)(\Psi) = \ckh(\wh(\Sigma_1))(\Phi) = 1$. This immediately implies that $\Sigma_0^T$ and $\Sigma_1^T$ are not smoothly equivalent rel.\ boundary. By Proposition~\ref{prop:ktgenus}, we know $K_T$ has trivial symmetry group, so we further conclude that there is no diffeomorphism of $B^4$ taking $\Sigma_0^T$ to $\Sigma_1^T$.
\end{proof}

\begin{figure}
	\centering\def\svgwidth{.78\linewidth}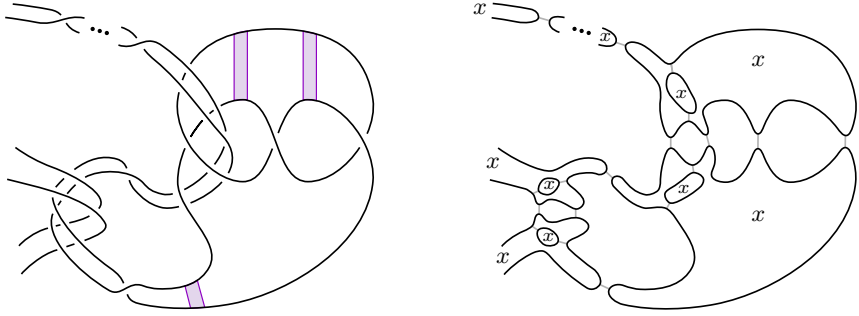
	\caption{Left: bands giving a movie of $\Sigma_1^T$ near $T$. Right: a labeled smoothing of the corresponding tangle near $T$}
	\label{fig:band_sum_phi}
\end{figure}

Interestingly, it is not clear how to perform analogous computations of knot Floer cobordism maps that obstruct two Seifert surfaces from becoming smoothly isotopic when pushed into $B^4$. Consider the following proposition, which follows immediately from functoriality and the grading shift calculations of Juh\'{a}sz--Marengon  \cite{juhaszmarengon}.

\begin{proposition}\label{prop:hfk}
If $S$ is a Seifert surface for a knot $J$ in $S^3$ such that $g(S)> g_3(J)$ and is decorated so that either the $z$-region or $w$-region is a bigon, then $S$ induces a trivial map on $\widehat{\hfk}(J)$. Moreover, any Seifert surface for another knot $J'$ into which $S$ embeds also induces a trivial map on $\widehat{\hfk}(J')$. 
\end{proposition}

In particular, this means that the usual cobordism maps on $\widehat{HFK}$ with trivial decorations cannot distinguish Seifert surfaces that are not of minimal Seifert genus. In contrast, our examples in Theorem~\ref{thm:khovanov} are distinctly {\textsl{not}} minimal genus Seifert surfaces. The surfaces in Theorem~\ref{thm:exotic} are minimal genus, but contain subsurfaces that are not minimal genus (specifically, the surfaces from Theorem~\ref{thm:khovanov}). Thus, we cannot {\textsl{easily}} use the knot Floer cobordism maps to distinguish two Seifert surfaces in $B^4$. In order to do so, one would either have to work with nontrivial decorations, find some new way of producing minimal genus examples, or perhaps use a further refinement of knot Floer homology (as opposed to $\widehat{\hfk}$).

\section{Obstructions from branched covers}
\label{sec:branched}

\subsection{A  topological counterexample} \label{section:topological} 

Let $\Sigma_0$ and $\Sigma_1$ be the surfaces of Figure~\ref{fig:46band}, with interiors pushed slightly into $B^4$.

\begin{theorem}\label{thm:doublecover}
Let $X_i$ be the double cover of $B^4$ branched along $\Sigma_i$ for $i=0,1$. The manifolds $X_0$ and $X_1$ are not homeomorphic.
\end{theorem}

Theorem~\ref{thm:doublecover} implies Theorem~\ref{thm:ambient}, since a locally flat isotopy from $\Sigma_0$ to $\Sigma_1$ would induce a homeomorphism from $X_0$ to $X_1$.

\begin{proof}[Proof of Theorem~\ref{thm:doublecover}]
In Figure~\ref{fig:doublecover}, we illustrate simple closed curves $a_0, b_0$ on $\Sigma_0$ and $a_1, b_1$ on $\Sigma_1$ that span the first homology. The resulting Seifert matrices (after appropriately orienting) are
\[V_0:=\left(\begin{matrix}-6&-2\\-3&-2\end{matrix}\right) \text{ for $\Sigma_0$} \qquad\text{ and }\qquad V_1:=\begin{pmatrix}-6&\phantom{-}0\\\phantom{-}1&-1\end{pmatrix}\text{ for $\Sigma_1$.}\]

The intersection form of $H_2(X_i;\mathbb{Z})$ is given by $V_i+V_i^T$, where the basis comes from doubling relative homology classes in $B^4\setminus\nu(\Sigma_i)$ bounded by $a_i$ and $b_i$ \cite{kauffman}. One can see this explicitly in the handle diagrams of $X_0$ and $X_1$ in the bottom of Figure~\ref{fig:doublecover}, obtained via the procedure of \cite[Section 2]{kirbyakbulut}. In particular, $X_1$ contains a locally flat embedded oriented surface $F$ whose Euler number is $-2$. On the other hand, the intersection form on $H_2(X_0;\mathbb{Z})$ is \[V_0+V_0^T=\begin{pmatrix}-12&-5\\-5&-4\end{pmatrix}.\] Letting  $A_0$ and $B_0$ denote the corresponding generators of $H_2(X_0;\mathbb{Z})$, we see that
\begin{align*}
    (mA_0+nB_0)\cdot(mA_1+nB_1)&=-12m^2-4n^2-10mn\\&=-2(6m^2+2n^2+5mn).\end{align*}
If there were a homeomorphism from $X_1$ to $X_0$, then the image of $F$ would have Euler number $\pm 2$, implying that there is an integer solution to $$6m^2+2n^2+5mn=\pm 1.$$
Solving for $n$, the discriminant of this equation is \[(5m)^2-4\cdot 2\cdot(6m^2\mp 1)=-23m^2\pm8.\] We must have this term be nonnegative to have a real solution. Since $m$ is an integer, we obtain $m=0$. This yields $2n^2=\pm 1$, a contradiction.
\end{proof}

\begin{figure}
       \labellist
\small\hair 2pt
\pinlabel $\Sigma_0$ at 10 250
\pinlabel $\textcolor{MutedRed}{a_0}$ at 10 180
\pinlabel $\textcolor{MutedBlue}{b_0}$ at 180 180
\pinlabel $M_0$ at 5 90
\pinlabel $\textcolor{MutedRed}{-12}$ at 295 15
\pinlabel $\textcolor{MutedBlue}{-2}$ at 295 65

\pinlabel $\Sigma_1$ at 215 250
\pinlabel $\textcolor{MutedRed}{a_1}$ at 245 115
\pinlabel $\textcolor{MutedBlue}{b_1}$ at 382 180
\pinlabel $M_1$ at 210 90
\pinlabel $\textcolor{MutedRed}{-12}$ at 93 -5
\pinlabel $\textcolor{MutedBlue}{-4}$ at 93 80
\endlabellist
    \includegraphics[width=110mm]{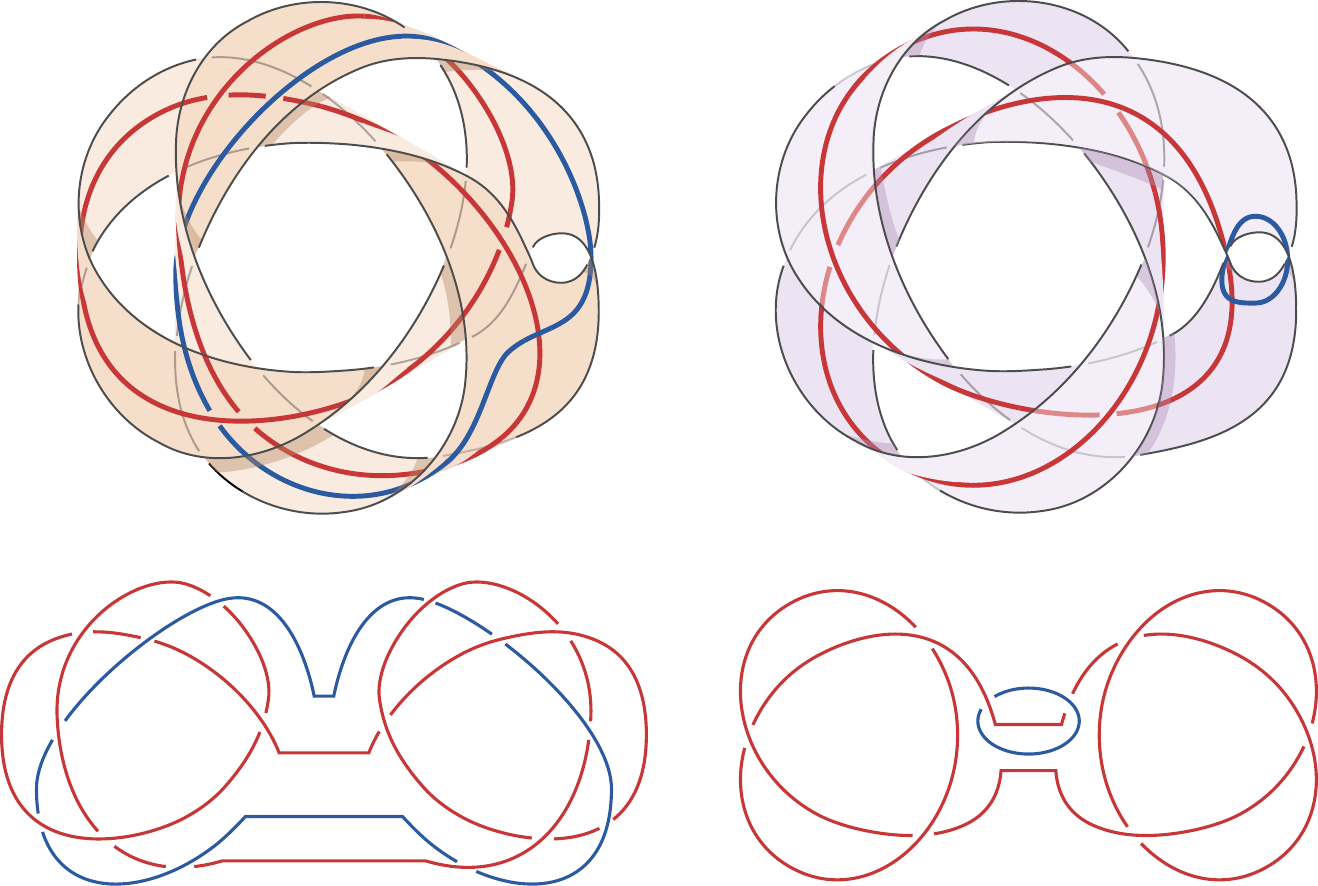}
    \caption{Top: a symplectic basis on each of $\Sigma_0$ and $\Sigma_1$. Bottom: handle diagrams for the double branched covers $X_0$ and $X_1$.}
    
    \label{fig:doublecover}
\end{figure}
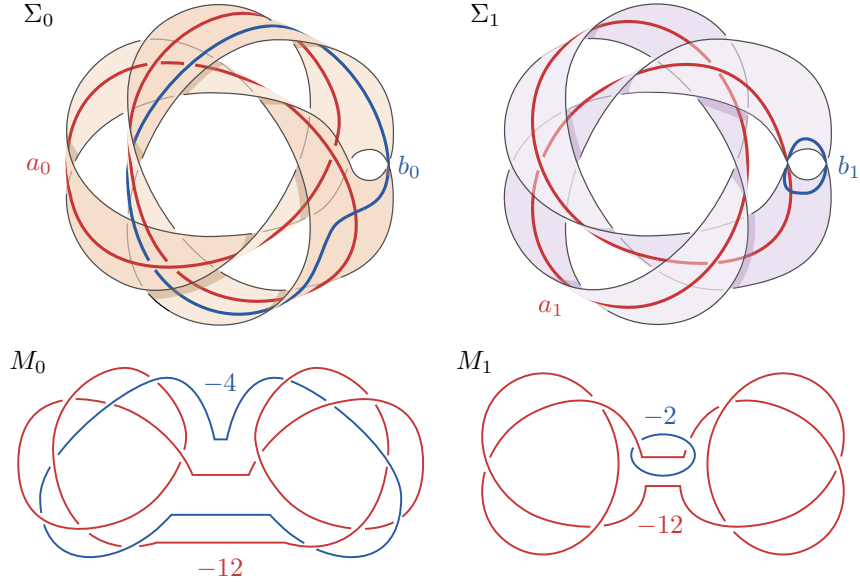

\begin{remark}
We find the argument of Theorem~\ref{thm:doublecover} striking in that its methods are extremely elementary and have been known to topologists for fifty years, yet are sufficient to answer a question in the literature that has been open for forty years.
\end{remark}

\begin{remark}
As shown in Figure~\ref{fig:union}, the union of the surfaces $\Sigma_0,\Sigma_1$ in $S^3$ is a standard genus-2 surface. Thus, even though $\Sigma_0,\Sigma_1$ are not topologically isotopic in $B^4$, their union $(B^4,\Sigma_0)\cup\overline{(B^4,\Sigma_1)}$ forms a smoothly unknotted closed surface in $S^4$.   It seems unlikely that gluing together modifications of the examples in this paper could produce oriented exotically knotted closed surfaces in $S^4$. 
\end{remark}

\begin{remark}
   A tube can be added to each of $\Sigma_0$ and $\Sigma_1$ to obtain a pair of genus-1 surfaces that are isotopic rel boundary in $B^4$. (Or in other words, $\Sigma_0$ and $\Sigma_1$ have {\emph{stabilization distance}} one.) We illustrate this in Figure~\ref{fig:stabilizationdistance}; here we draw a ribbon surface in $B^4$ by drawing its boundary knot and bands attached to that knot representing index-1 points of the surface with respect to radial height.
    
    In the top row of Figure~\ref{fig:stabilizationdistance}, we show that two particular bands attached to a link are isotopic. (Note that we allow the ends of the bands to slide along the link, and we allow the two bands to pass through each other). This link is obtained by surgering $K$ along the other three bands in the bottom two illustartions. We conclude that the surfaces illustrated in the bottom of Figure~\ref{fig:stabilizationdistance} are smoothly isotopic rel.\ boundary in $B^4$. The surface on the left is obtained from $\Sigma_0$ by attaching a single tube and the surface on the right is obtained from $\Sigma_1$ by attaching a single tube.
\end{remark}

\begin{figure}
    \centering
    \includegraphics[width=.85\linewidth]{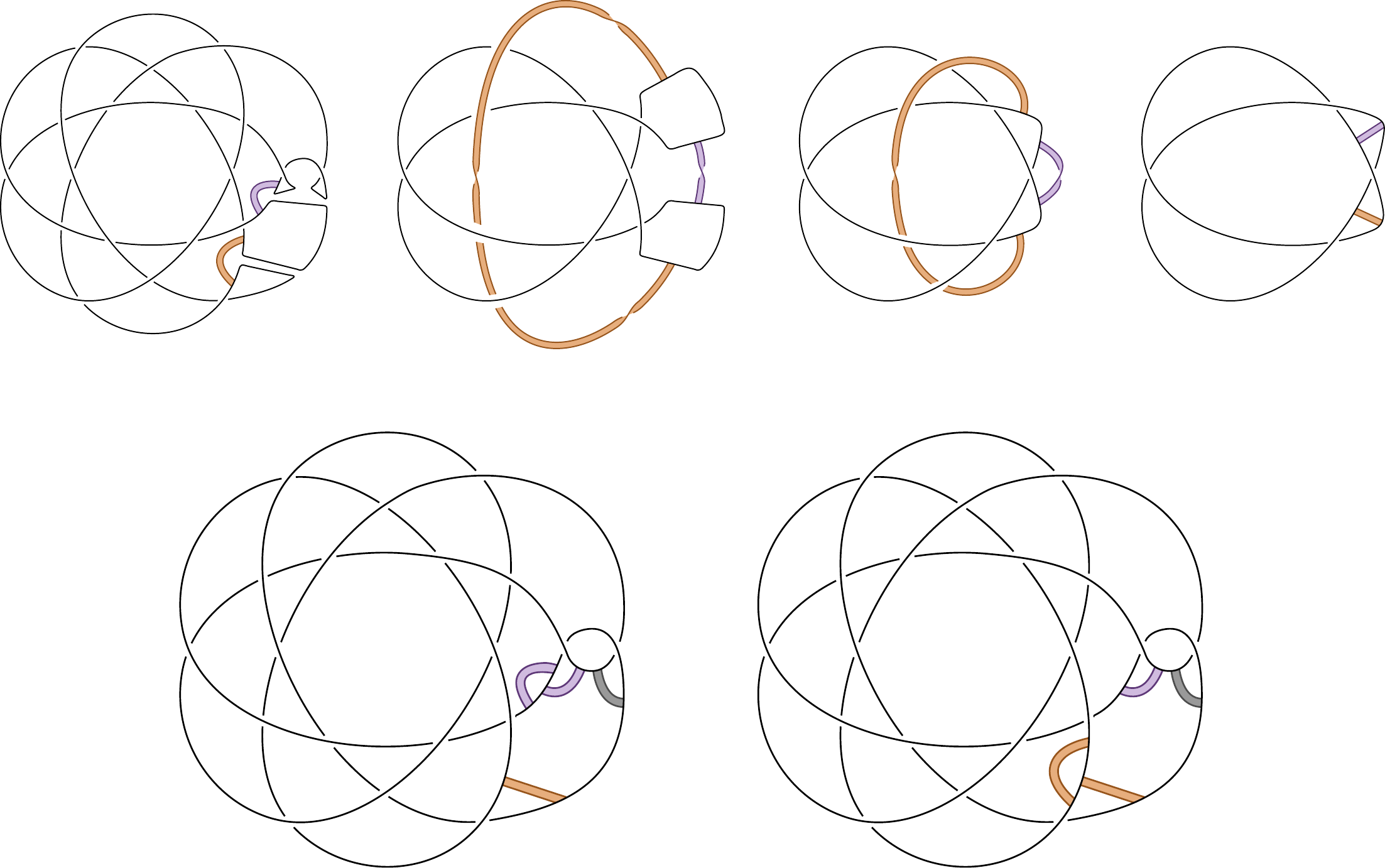}
    \captionsetup{width=.85\linewidth}
    \caption{Top row, from left to right: We draw two bands attached to a link. We produce an isotopy from one band to the other. Bottom row: We conclude that the two indicated surfaces are smoothly isotopic rel.\ boundary in $B^4$.}
    \label{fig:stabilizationdistance}
\end{figure}

\begin{remark}
    We can modify $\Sigma_0$ and $\Sigma_1$ to produce pairs of non-orientable surfaces; see Remark~\ref{rem:nonorientgenus} or simply consider the specific pair $\Sigma_0'$ and $\Sigma_1'$ in Figure~\ref{fig:nonorient}. The 2-fold branched covers $X_0',X_1'$ of $B^4$ branched along $\Sigma_0',\Sigma_1'$ (after pushing their interiors into the interior of $B^4$) have intersection forms on $H_2$ respectively presented by \[\begin{pmatrix}-12&-5\\-5&-3\end{pmatrix}\qquad\text{ and }\qquad\begin{pmatrix}-12&\phantom{-}1\\\phantom{-}1&-1\end{pmatrix},\] which can be easily distinguished by e.g.\ the same argument as in Theorem~\ref{thm:doublecover} (there is no embedded, oriented locally flat surface in $X_0'$ with Euler number $-1$, while there is in $X_1'$). Therefore, the surfaces $\Sigma_0'$ and $\Sigma_1'$ are not topologically isotopic in $B^4$.
\end{remark}

\begin{figure}
    \centering
    \labellist
    \small\hair 2pt
\pinlabel $\Sigma_0'$ at 5 140
\pinlabel $\Sigma_1'$ at 212 140
\endlabellist
    \includegraphics[width=90mm]{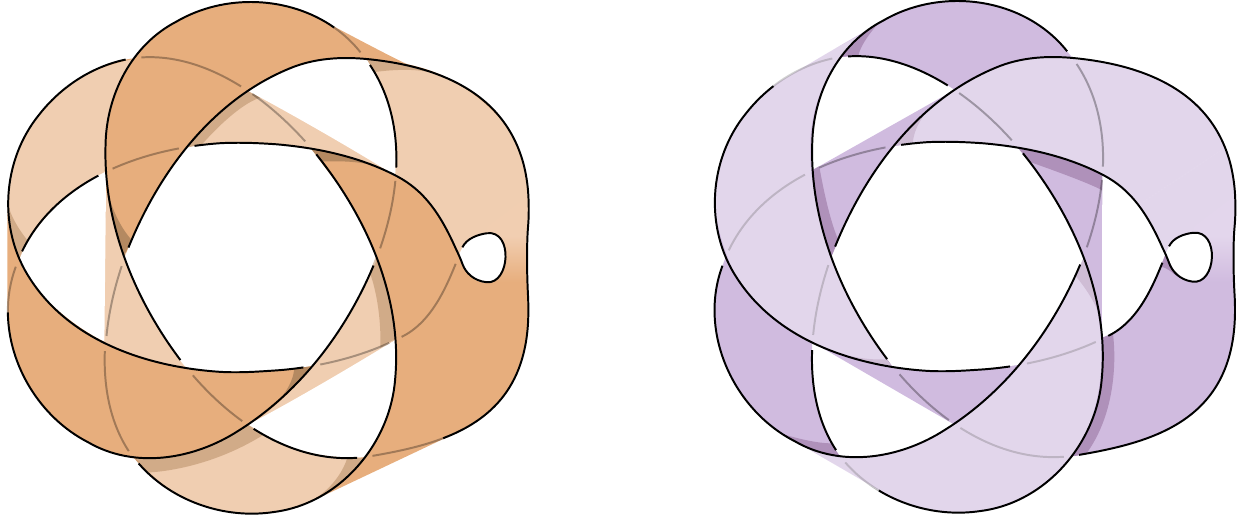}
    \caption{Non-orientable surfaces that are not topologically isotopic in $B^4$.}
    \label{fig:nonorient}
\end{figure}

As an aside, we note that the intersection form argument of Theorem~\ref{thm:doublecover} can imply other interesting results about the structure of a surface in $B^4$. For example, consider the M\"{o}bius band $M$ bounded by the knot $\tt{8_{20}}$ illustrated in Figure~\ref{fig:mobius} (left). The double cover $X$ of $B^4$ branched along $M$ admits a handle decomposition consisting of a 0-handle and a 2-handle attached along a right-handed trefoil with framing $9$. We conclude from the intersection form on $H_2(X;\mathbb{Z})$ that $X$ does not admit a $\smash{ \mathbb{CP}^2}$ or $\smash{ \overline{\mathbb{CP}}^2}$ summand, and hence $M$ does not decompose topologically as a connected sum of an unknotted projective plane and a slice disk for $\tt{8_{20}}$. Of course, this would be obvious in the case that $\partial M$ were not slice -- but given that $\tt{8_{20}}$ is slice, this is not so clear and might be viewed as a negative answer to a relative version of the Kinoshita conjecture, which asks whether every projective plane in $S^4$ decomposes as a connected sum of a knotted 2-sphere and an unknotted projective plane.

\begin{figure}
    \centering
    \labellist
    \pinlabel \textcolor{orangenew}{9} at 440 140 
    \endlabellist
    \includegraphics[width=100mm]{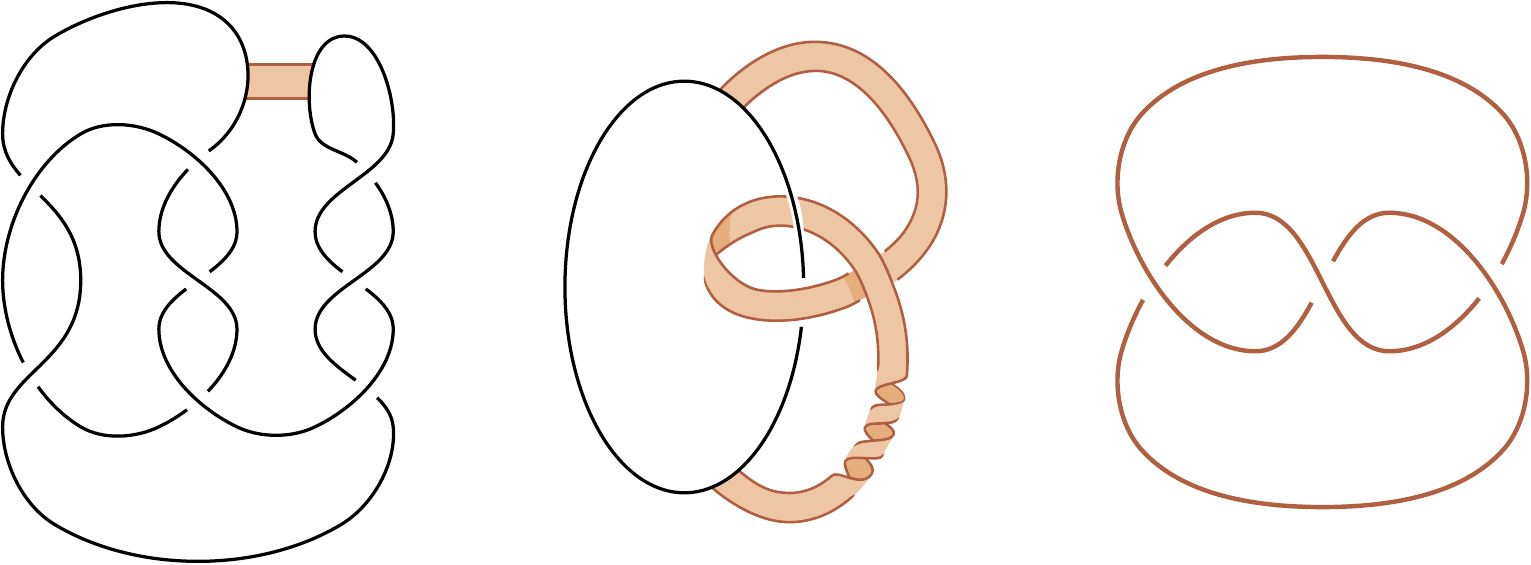}
    \caption{Left: A M\"obius band $M$ in $B^4$ bounded by $\tt{8_{20}}$, consisting of a disk bounded by the drawn curve together with the indicated nonorientable band. Middle: we isotope the left picture. Right: we obtain the double cover of $B^4$ branched along $M$ via the procedure of \cite{kirbyakbulut} (see also \cite[Section 6.3]{gompfstipsicz} or \cite[Chapter 11]{akbulutbook}).} 
    
    \label{fig:mobius}
\end{figure}

\begin{corollary}\label{cor:mobius}
Any 1-minimum ribbon M\"obius band $M$ properly embedded in $B^4$ whose boundary is a knot $J$ with $\det(J)>1$ does not decompose topologically as a connected sum of a slice disk for $J$ and an unknotted projective plane.
\end{corollary}

Corollary~\ref{cor:mobius} applies to the M\"{o}bius band for ${\tt{8_{20}}}=P(2,-3,3)$ featured in Figure~\ref{fig:mobius}, as well as the analogous M\"{o}bius band for the slice pretzel knot $P(2n,-2n-1,2n+1)$ for any $n\ge 1$. (Note that this knot has determinant $4n^2+4n+1$.) Each of these pretzel knots admits a standard ribbon disk obtained by a band move across the even strand. It is a relatively easy exercise to check that gluing this disk to the described M\"obius band yields a smoothly unknotted projective plane in $S^4$.

\begin{proof}[Proof of Corollary~\ref{cor:mobius}]
Since $M$ has one minimum, one saddle, and no maxima, the 2-fold cover $X$ of $B^4$ branched along $M$ can be built from a single 0-handle and 2-handle. The boundary of $X$ is the 2-fold cover of $S^3$ branched over $J$, which has first homology of order $\det(J)$. This must then be the framing of the 2-handle up to sign, so the intersection form on $H_2(X;\mathbb{Z})=\mathbb{Z}$ is $[\pm\det(J)]$. We conclude that $X$ does not admit a $\mathbb{CP}^2$ or $\smash{\overline{\mathbb{CP}}^2}$ summand and hence $M$ does not admit an unknotted projective plane summand.
\end{proof}

\subsection{Infinite families of surfaces up to isotopy rel boundary}\label{subsec:infinite_family}

In this section, we prove Theorems~\ref{thm:infinite} and \ref{thm:infiniteunlink} by constructing two infinite families of surfaces that are not isotopic rel.\ boundary. In the first case, the surfaces are disconnected but not even topologically isotopic rel.\ boundary, while in the second case the surfaces are connected but we give only a smooth obstruction to isotopy rel.\ boundary.

We begin with Theorem \ref{thm:infiniteunlink}. Consider the surface $S_0$ in Figure~\ref{fig:infinitefamilybands}. This surface is a split union of a disk and an annulus; we illustrate a torus $T$ intersecting $S_0$ in two circles. Let $S_n$ be obtained from $S_0$ by performing $n$ meridional twists to $S_0$ about $T$.

\begin{figure}
    \centering
        \labellist
\small\hair 2pt
\pinlabel {\textcolor{darkgray}{$S_0$}} at -5 105
\pinlabel{\textcolor{blue}{$T$}} at 33 46
\pinlabel{\textcolor{MutedRed}{$b$}} at 258 78
\pinlabel{\textcolor{MutedRed}{$b_n$}} at 355 98
\pinlabel{$n$} at 347 55
\pinlabel{\scalebox{.6}{$\boldsymbol{\hdots}$}} at 355 79
\endlabellist
    \includegraphics[width=125mm]{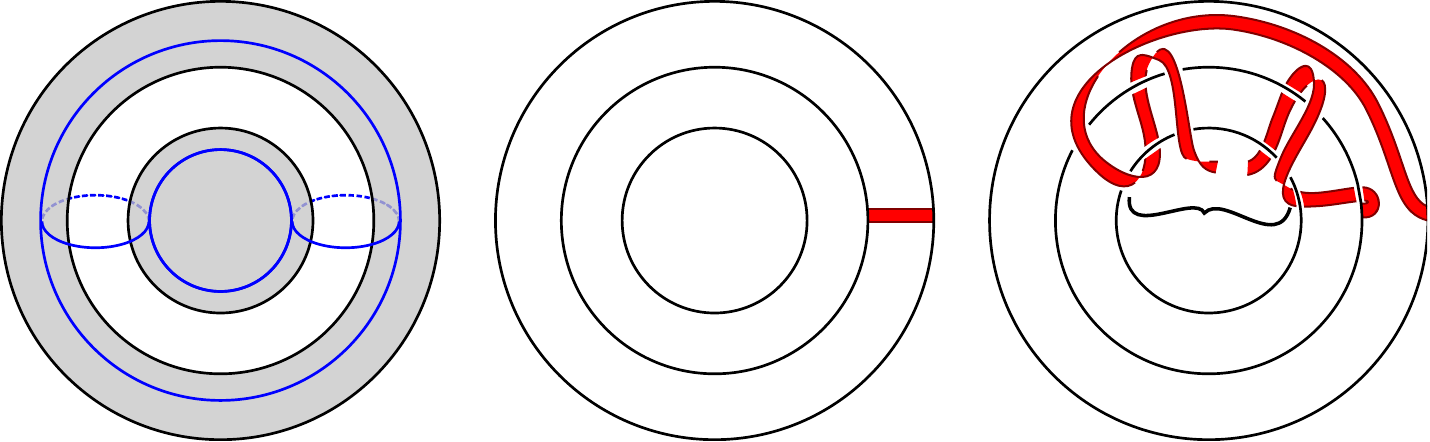}
    \caption{Left: the surface $S_0$ and  the torus $T$. Middle: the band $b$. Right: the band $b_n$.}
    \label{fig:infinitefamilybands}
\end{figure}
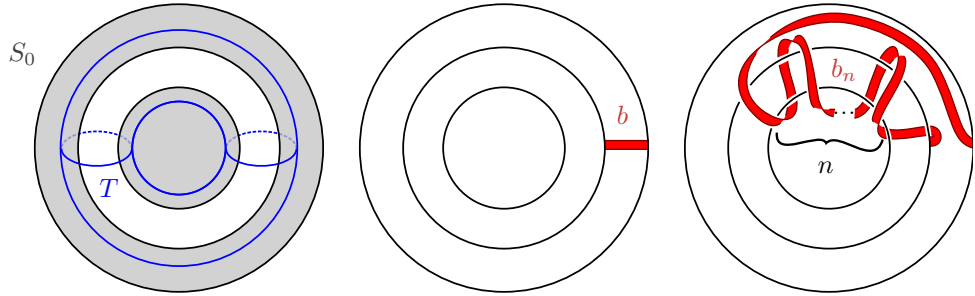

\begin{proposition}
For $n\neq 0$, the surfaces $S_n$ and $S_0$ are not topologically isotopic rel.\ boundary.
\end{proposition}

\begin{proof}
First note that $S_n$ is freely isotopic $S_0$. Indeed, $S_n$ is obtained from $S_0$ by applying $n$ meridional twists $\tau$ along the torus $T$ in the complement of the unlink $\partial S_0$. As mentioned in \S\ref{subsec:spinning}, the  diffeomorphism $\tau$ of $S^3$ (which is supported near $T$) is isotopic to the identity via an isotopy supported near the solid torus that $T$ bounds in Figure~\ref{fig:infinitefamilybands} (left).  Push the interiors of $S_0$ and $S_n$ into $B^4$, and extend $\tau$ over $B^4$ to preserve the relation $S_n=\tau^n(S_0)$.

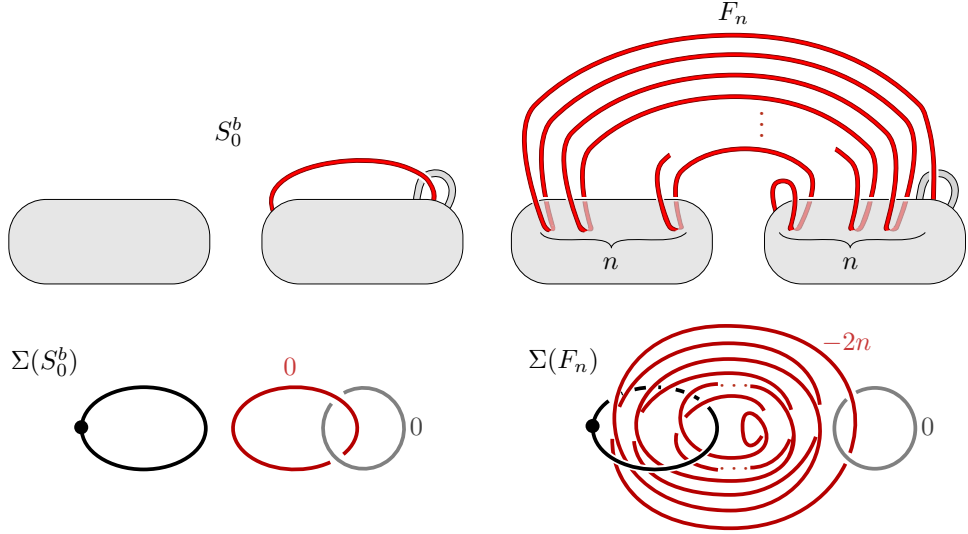
\begin{figure}
    \centering
        \labellist
\small\hair 2pt
\pinlabel $S^b_0$ at 117.5 215 
\pinlabel $F_n$ at 385 278
\pinlabel{\textcolor{MutedRed}{$0$}} at 150 90
\pinlabel{\textcolor{darkgray}{$0$}} at 217 58
\pinlabel{\textcolor{MutedRed}{$-2n$}} at 445 103
\pinlabel{\textcolor{darkgray}{$0$}} at 488 58
\pinlabel{$n$} at 320 145
\pinlabel{$n$} at 447 145
\small\hair 2pt
\pinlabel $\Sigma(S^b_0)$ at 20 93
\pinlabel $\Sigma(F_n)$ at 295 93
\pinlabel{\textcolor{BrickRed}{\large$\boldsymbol{\vdots}$}} at 400 223
\pinlabel{\textcolor{BrickRed}{$\cdots$}} at 386.5 80
\pinlabel{\textcolor{BrickRed}{$\cdots$}} at 386.5 35.5
\endlabellist
\vspace{.1in}
    \includegraphics[width=130mm]{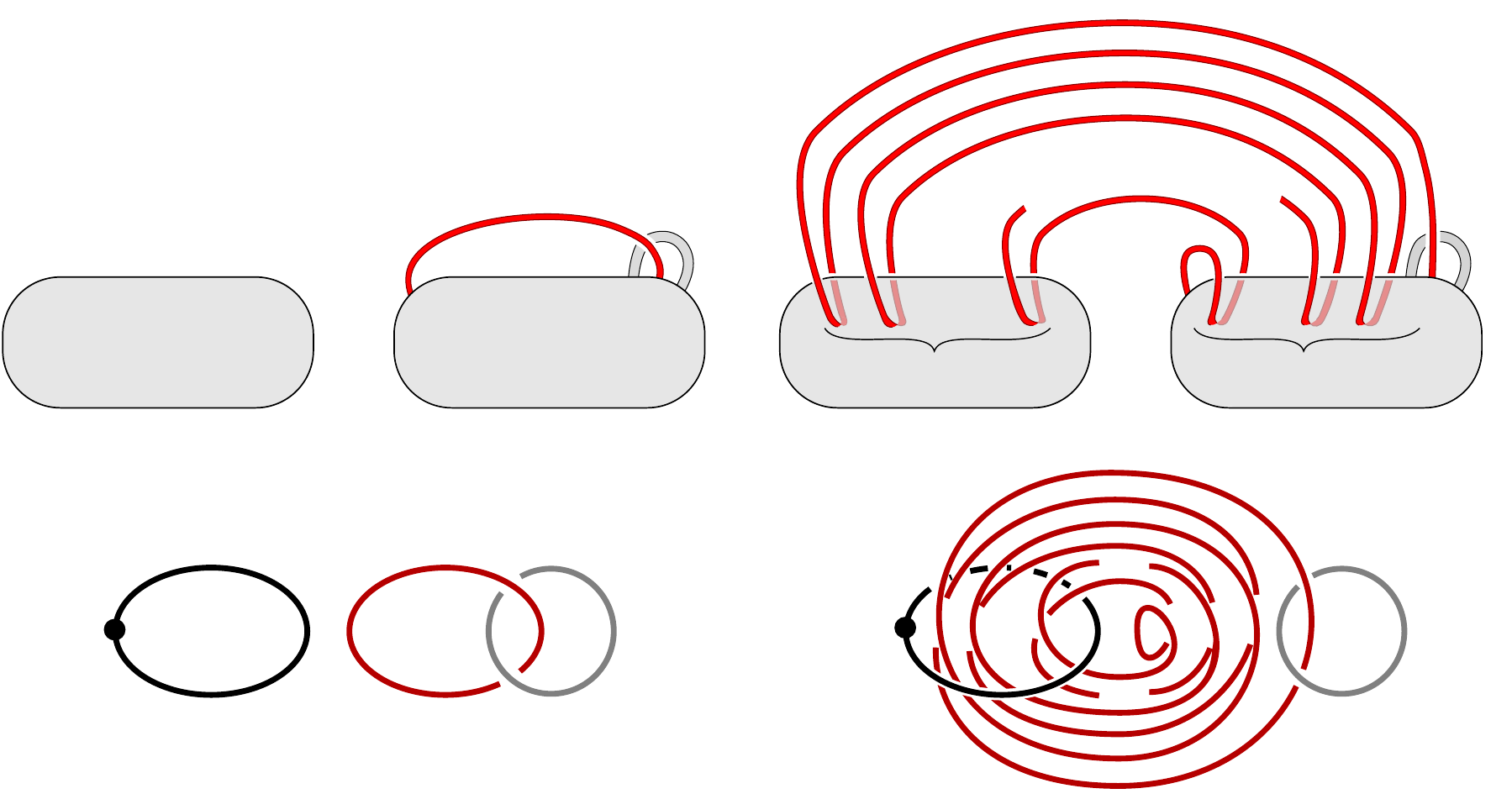}
    \caption{In the top row, we draw $S^b_0$ and $F_n$ as ribbon surfaces in $B^4$. If $S_0$ and $S_n$ are isotopic rel.\ boundary, then $S^b_0$ and $F_n$ are freely isotopic. In the bottom row, we perform the algorithm of \cite{kirbyakbulut} to obtain Kirby diagrams of the respective 2-fold branched covers $\Sigma(S^b_0)$ and $\Sigma(F_n)$.}
    \label{fig:infinitefamilycovers}
\end{figure}

Towards a contradiction, suppose $S_n$ is isotopic to $S_0$ rel.~boundary in $B^4$. Consider the enlarged surfaces $S_0^b,S_n^b \subset B^4$ obtained by attaching the band $b \subset S^3$ shown in Figure~\ref{fig:infinitefamilybands} to each of $S_0,S_n \subset B^4$, respectively. Note that the isotopy rel.~boundary between $S_0$ and $S_n$ induces an isotopy between $S_0^b$ and $S_n^b$. Rather than study $S_n^b$ directly, it will be simpler to consider its pullback $F_n=\tau^{-n}(S_n^b)$ which, by construction, is isotopic to $S_n^b$. Observe that $$F_n=\tau^{-n}(S_n \cup b)=\tau^{-n}(S_n)\cup \tau^{-n}(b)=S_0 \cup \tau^{-n}(b),$$
so $F_n$ is equivalently obtained from $S_0$ by attaching the band $b_n=\tau^{-n}(b)$ from Figure~\ref{fig:infinitefamilybands} (right). In Figure~\ref{fig:infinitefamilycovers}, we redraw $S_0^b$ and $F_n$ and produce Kirby diagrams of the 2-fold branched covers of $B^4$ over $S_0^b$ and $F_n$. We see that $$H_1(\Sigma(S^b_0);\mathbb{Z})=\mathbb{Z}\qquad\text{ while }\qquad H_1(\Sigma(F_n);\mathbb{Z})=\mathbb{Z}_{2|n|}.$$ This contradicts $F_n$ being freely isotopic to $S^b_0$, hence $S_0$ and $S_n$ cannot be isotopic rel.~boundary in $B^4$.
\end{proof}

The next corollary implies Theorem \ref{thm:infiniteunlink}.
\begin{corollary}
For $m\neq n$, the surfaces $S_m$ and $S_n$ are not topologically isotopic rel.\ boundary.
\end{corollary}
\begin{proof}
There is a homeomorphism of $B^4$ that maps $S_m$ and $S_n$ to $S_0$ and $S_{n-m}$, respectively. Since $S_0$ and $S_{n-m}$ are not topologically isotopic rel.\ boundary, neither are $S_m$ and $S_n$.
\end{proof}

We now proceed to the proof of Theorem \ref{thm:infinite}. The construction  similarly relies on torus twisting, but we can arrange for the surfaces in question to be connected at the cost of complicating the boundary link. To that end, let $S=S_0$ be the surface in Figure~\ref{fig:dbc}, which is isotopic to the surface $S_0$ from Example~\ref{ex:infinite} (originally shown in the center of Figure~\ref{fig:infinite}). Example~\ref{ex:infinite} also introduced surfaces $S_n$ obtained from $S_0$ by twisting $n$ times about the torus $T$ pictured in Figure~\ref{fig:infinite} (right). Our first step in distinguishing these Seifert surfaces up to smooth isotopy rel.~boundary in $B^4$ will be to show that the double branched cover $M_n=\Sigma_2(B^4,S_n)$ is obtained from $M=\Sigma_2(B^4,S)$ by  surgery along the torus $T'$ depicted in Figure~\ref{fig:dbc} (right).

\begin{proposition}\label{prop:cover}
    The 4-manifold $M$ shown in the left of Figure~\ref{fig:dbcsurgery} is the double branched cover of $B^4$ along $S=S_0$.
\end{proposition}
\begin{proof}
In Figure \ref{fig:dbc} we obtain a Kirby diagram for the double branched cover of $B^4$ along $S$, using the algorithm of \cite{kirbyakbulut}. In Figure \ref{fig:dbcmoves}, we show how to perform Kirby moves to this diagram to arrive at Figure \ref{fig:dbcsurgery} (left).
\end{proof}

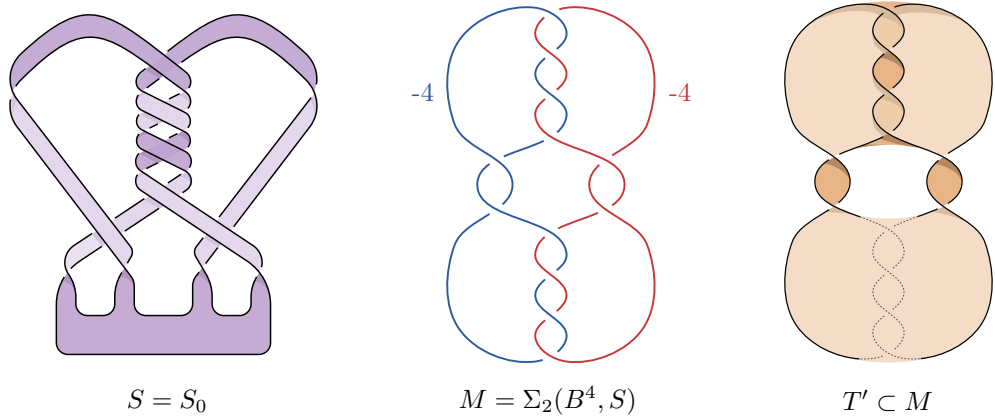
\begin{figure}\center
       \labellist
\small\hair 2pt

\pinlabel \textcolor{MutedRed}{-$4$} at 446 180
\pinlabel \textcolor{MutedBlue}{-$4$} at 275 180

\pinlabel $S=S_0$ at 105 -25
\pinlabel $M=\Sigma_2(B^4,S)$ at 358 -25
\pinlabel ${T' \subset M}$ at 584 -25

\endlabellist

\includegraphics[width=.9\linewidth]{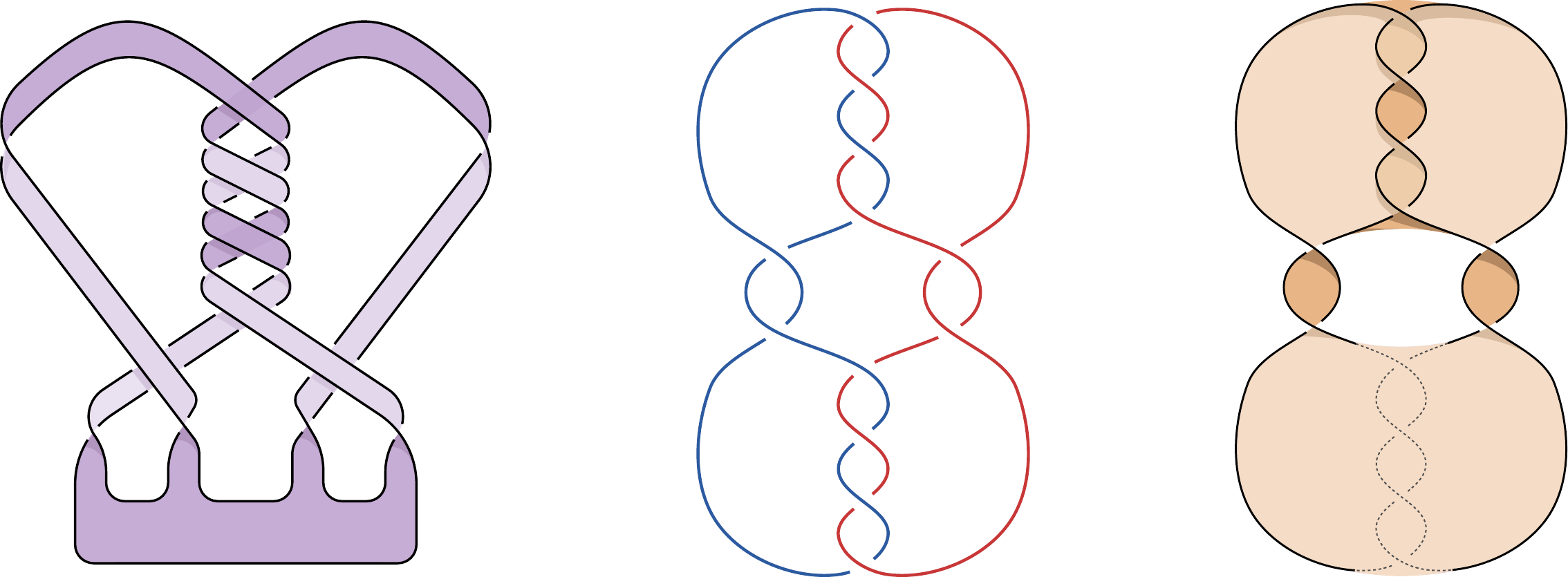}

\vspace{.2in}

    \captionsetup{width=.8\linewidth}
    \caption{
    Left: An alternative diagram for the surface from Example~\ref{ex:infinite}. Center: A Kirby diagram for the 2-fold branched cover of $B^4$ along $S$. 
    Right: A torus $T'$ of square zero, consisting of a genus-1 surface in $S^3$ and the core disks of the two 2-handles in the middle diagram. In Proposition \ref{prop:toruscover}, we show that the lift of $T$ in $S^3\setminus L$ consists of two parallel copies of $T'$.}
    
    \label{fig:dbc}
\end{figure}

\begin{figure}
\center
    
\labellist
\small\hair 2pt

\pinlabel \textcolor{MutedRed}{-$4$} at 262 200
\pinlabel \textcolor{MutedBlue}{$0$} at 10 110
\pinlabel \textcolor{gray}{-$2$} at 49 177
\pinlabel \textcolor{gray}{-$2$} at 245 15
\pinlabel $M=M_{(0,0,1)}$ at 120 -35

\pinlabel \textcolor{MutedRed}{-$4$} at 640 200
\pinlabel \textcolor{MutedBlue}{$0$} at 363 34
\pinlabel \textcolor{gray}{-$2$} at 425 177
\pinlabel \textcolor{gray}{-$2$} at 623 15
\pinlabel $M_{(0,1,0)}$ at 492 -35

\pinlabel \textcolor{MutedBlue}{$0$} at 755 120
\pinlabel {$T^2 \times D^2$} at 854 -35
\pinlabel \textcolor{gray}{$\alpha$} at 872 200
\pinlabel \textcolor{gray}{$\gamma$} at 947 100
\pinlabel \textcolor{gray}{$\beta$} at 982 50
\endlabellist

\includegraphics[width=\linewidth]{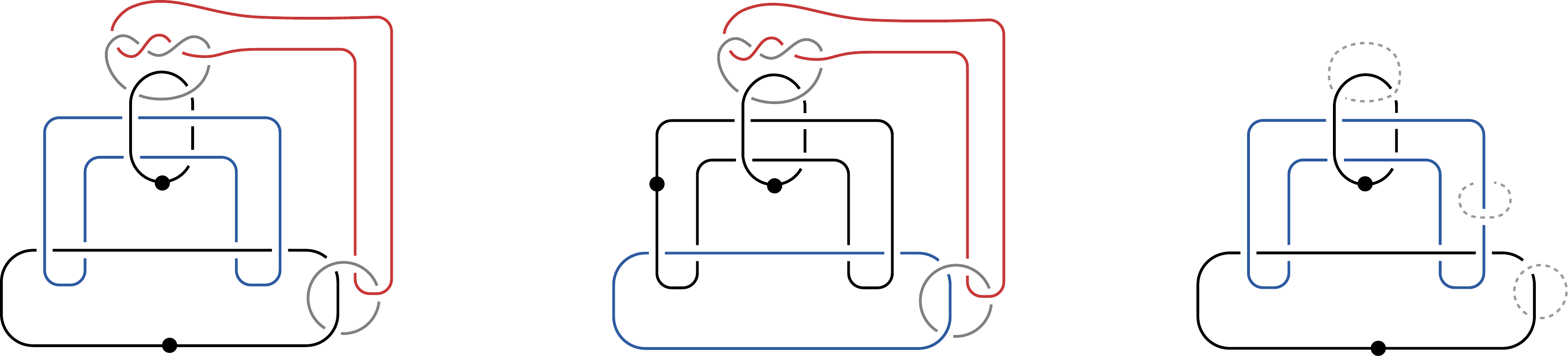}
   \vspace{.1in}

    \caption{Left: a 4-manifold $M$. Right: A copy of $T^2\times D^2$ contained in $M$, with curves $\alpha,\beta,\gamma$ in $\partial T^2\times D^2$ indicated. The torus $T^2\times 1$ is equivalent $T'$ from Figure \ref{fig:dbc}. Here, $\gamma=\pt\times\partial D^2$, so surgering $M$ along $T^2\times 0$ gluing a meridian disk to $(0,0,1)$ in $(\alpha,\beta,\gamma)$ coordinates is the trivial surgery (and hence $M$ can be identified with $M_{(0,0,1)}$. Middle:  A diagram of $M_{(0,1,0)}$.}
    
    \label{fig:dbcsurgery}
\end{figure}
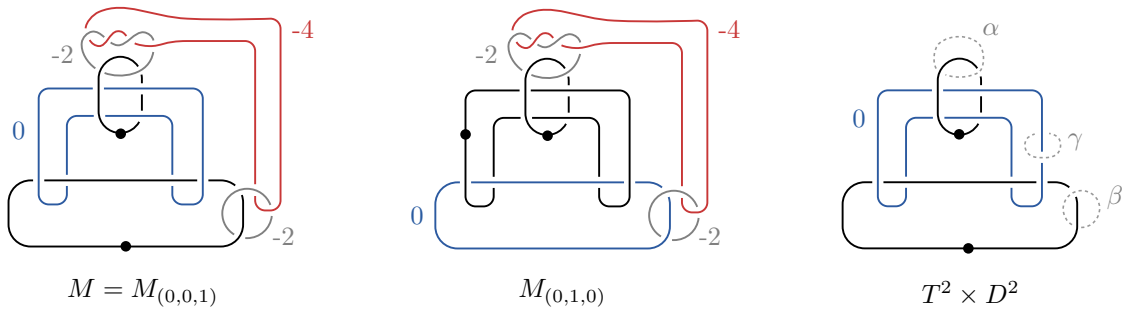

\begin{figure}
\center
\labellist\footnotesize
\pinlabel{\textcolor{MutedRed}{$-4$}} at 5 100
\pinlabel{\textcolor{MutedBlue}{$-4$}} at -4 129
\pinlabel{\textcolor{MutedRed}{$-4$}} at 112 93
\pinlabel{\textcolor{MutedBlue}{$0$}} at 106 130
\pinlabel{\textcolor{MutedRed}{$-4$}} at 279 90
\pinlabel{\textcolor{MutedBlue}{$0$}} at 225 131
\pinlabel{\textcolor{MutedRed}{$-4$}} at 133 10
\pinlabel{\textcolor{MutedBlue}{$0$}} at 73 60
\pinlabel{\textcolor{MutedRed}{$-4$}} at 239 10
\pinlabel{\textcolor{MutedBlue}{$0$}} at 172 60
\pinlabel{\textcolor{gray}{$-2$}} at 211 3
\pinlabel{\textcolor{gray}{$-2$}} at 207 59
\pinlabel{$-2$} at 85 7 
\pinlabel{$-2$} at 98 40
\endlabellist
    \includegraphics[width=.9\linewidth]{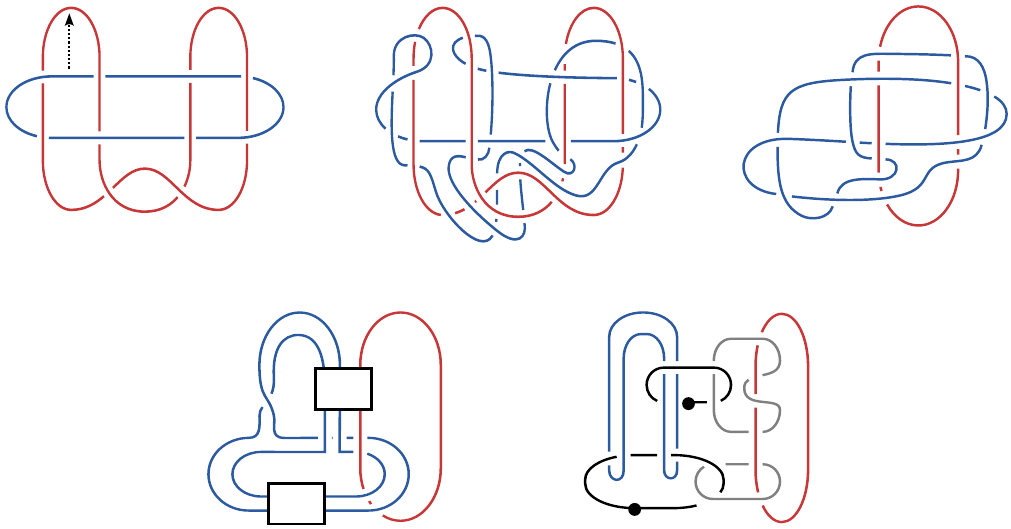}
        \captionsetup{width=.875\linewidth}
    \caption{Beginning with the diagram of $M$ on the top left (obtained from Figure~\ref{fig:dbc} by isotopy), we perform Kirby moves until arriving at the alternative diagram of $M$ on the bottom right, which is isotopic to the diagram of $M$ in Figure~\ref{fig:dbcsurgery}.}
    
    \label{fig:dbcmoves}
\end{figure}


Following the above Kirby moves also shows that the torus $T'$ depicted in Figure~\ref{fig:dbc} is isotopic to the obvious torus $T^2 \times 0$ in the copy of $T^2 \times D^2 \subset M$ depicted in Figure~\ref{fig:dbcsurgery}. The second manifold $M_{(0,1,0)}$  in Figure \ref{fig:dbc} is obtained by surgery on this torus $T'$,  and the surgery is of type $(0,1,0)$ relative to the basis shown.

\begin{proposition}\label{prop:toruscover}
    The torus $T\subset S^3\setminus L$ of Figure \ref{fig:infinite} lifts to two copies of obvious torus $T'$ inside $M$ (after being pushed into the interior of $M$).
\end{proposition}
\begin{proof}
In $S^3$, the torus $T$ bounds a solid torus $V$ intersecting $L$ in two circles that are both cores of the solid torus. Then $V$ lifts to a copy of $T^2\times I$ in $\partial M$, so we conclude that the lift of $T$ to $M$ consists of two disjoint, parallel tori in $\partial M$. 

To identify these tori upstairs, one could  perform this lift explicitly and obtain parallel copies of the torus illustrated in the right of Figure~\ref{fig:dbc}, which is the torus visible in Figure~\ref{fig:dbcsurgery}. Instead, we will appeal to an indirect argument, as follows. 

In the diagram of $M$ given in Figure~\ref{fig:dbc}, the branched covering involution of $\partial M$ (whose quotient is $S^3$) corresponds to rotation about a horizontal axis. By considering the diagram of $T'$ in Figure \ref{fig:dbc} (right), we see that positive and negative pushoffs of $T'$ can be taken disjoint from the branch set and are exchanged by the covering involution. We conclude that they project to a single embedded torus $\hat{T}$ in $S^3\setminus L$. Since $T'$ is nonseparating in $\partial M$, there must be components of $L$ on each side of $\hat{T}$, and $\hat{T}$ cannot be boundary-parallel in $S^3\setminus\nu(L)$. Since $L$ is non-split, then $\hat{T}$ is essential. 

From Figure~\ref{fig:infinite}, it is straightforward to see that $S^3 \setminus L$ contains exactly two essential tori that are not boundary parallel (up to isotopy). One is the torus $T$ that encloses two components of $L$, and the second torus encloses a single component of $L$ which sits as a twisted Whitehead pattern knot inside in enclosed solid torus. The double branched cover of the solid torus over the Whitehead pattern knot is not $T^2 \times I$, so this second torus in $S^3 \setminus L$ does not lift to parallel copies of $T'$. We conclude that $\hat T$ is isotopic to $T$, as claimed.
\end{proof}

\begin{proposition}\label{prop:xpcover}
    The double branched cover $M_n$ of $B^4$ along $S_n$ is the 4-manifold $M_{(0,2n,1)}$ obtained from $M$ by $(0,2n,1)$-surgery along $T'$.
\end{proposition}

\begin{proof}
Recall that $S_n$ is obtained from $S_0$ by twisting $n$ times about $T$. By Proposition \ref{prop:toruscover}, the lift of $T$ to $M$ consists of two parallel copies of $T'$. Then $M_n$ is obtained from $M$ by performing degree-1 surgeries on each of two parallel copies of $T'$ in the direction of the twist. In our chosen coordinates, the direction is $n\beta$, so $M_n$ is obtained from $M$ by performing $(0,n,1)$ surgeries along two copies of $T'$.

\begin{figure}
    \centering
    \labellist
    \footnotesize
    \pinlabel{\textcolor{gray}{$\alpha$}} at 67 553
    \pinlabel{\textcolor{gray}{$\beta$}} at 145 520
    \pinlabel{\textcolor{gray}{$\gamma$}} at 130 560
    \pinlabel{\textcolor{MutedBlue}{$0$}} at 45 585
     \pinlabel{\textcolor{MutedBlue}{$0$}} at 250 585
     \pinlabel{\textcolor{MutedBlue}{$0$}} at 445 585
     \pinlabel{\textcolor{MutedRed}{$0$}} at 345 577
     \pinlabel{\textcolor{MutedRed}{$0$}} at 345 499
\pinlabel{\textcolor{black}{$\cup \text{ 3-handle}$}} at 445 502
     
        \pinlabel{\textcolor{MutedBlue}{$0$}} at 12 460
     \pinlabel{\textcolor{MutedBlue}{$0$}} at 201 310
     \pinlabel{\textcolor{MutedRed}{$0$}} at 112 480
     \pinlabel{\textcolor{MutedRed}{$0$}} at 112 374
     \pinlabel{$n$} at 79.5 420
     \pinlabel{$n$} at 142.5 420
     \pinlabel{\textcolor{MutedBlue}{$0$}} at 270 460
     \pinlabel{\textcolor{MutedBlue}{$0$}} at 465 460
     \pinlabel{$n$} at 338.25 420
     \pinlabel{$n$} at 401.25 420
        \pinlabel{\textcolor{MutedBlue}{$0$}} at 12 327
     \pinlabel{\textcolor{MutedBlue}{$0$}} at 207 462
     \pinlabel{$n$} at 79.5 287
     \pinlabel{$n$} at 142.5 287
     \pinlabel{\textcolor{MutedBlue}{$0$}} at 267 327
     \pinlabel{\textcolor{MutedBlue}{$0$}} at 457 328
     \pinlabel{$n$} at 338.5 287
     \pinlabel{$n$} at 401.5 287
      \pinlabel{\textcolor{MutedBlue}{$0$}} at 22 182
     \pinlabel{\textcolor{MutedBlue}{$0$}} at 209 196
     \pinlabel{$n$} at 79.5 155
     \pinlabel{$n$} at 142.5 155
     \pinlabel{\textcolor{MutedBlue}{$0$}} at 309 137
     \pinlabel{\textcolor{MutedBlue}{$0$}} at 457 196
     \pinlabel{$n$} at 327.5 155
     \pinlabel{$n$} at 390.5 155

      \pinlabel{\textcolor{MutedBlue}{$0$}} at 57 75
     \pinlabel{\textcolor{MutedBlue}{$0$}} at 190 75
     \pinlabel{$n$} at 59 35
     \pinlabel{$n$} at 122 35
     \pinlabel{\textcolor{MutedBlue}{$0$}} at 319 75
     \pinlabel{$2n$} at 338.5 35
    \endlabellist
    \includegraphics[width=\textwidth]{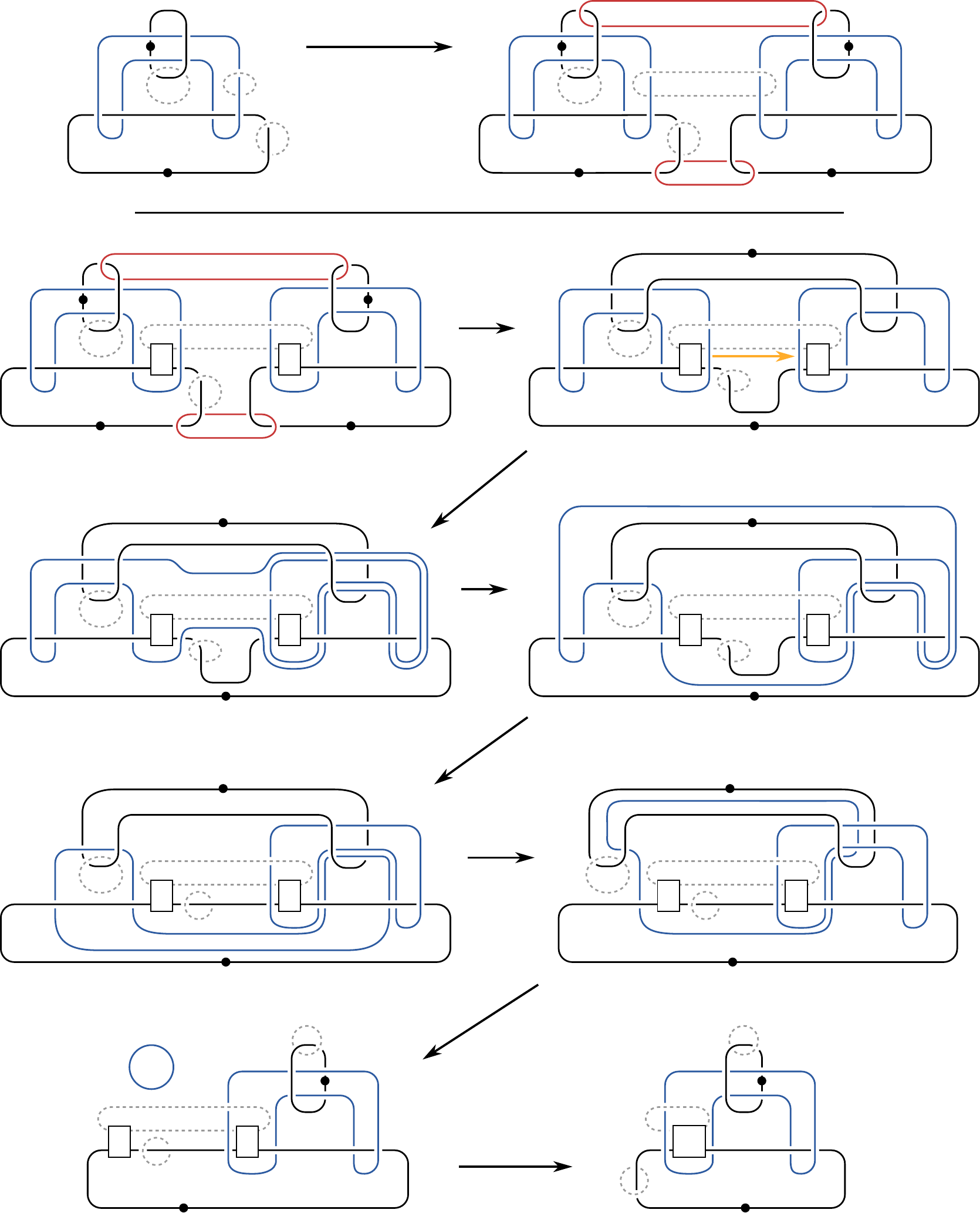}
    \captionsetup{width=.9\linewidth}
    \caption{Top left: a diagram of $T^2\times D^2$. Top right: we perform Kirby moves so as to make two copies of $T^2\times\{\pt\}$ easily visible. Second row and down, following arrows: we perform $(0,n,1)$ surgery on two parallel copies of $T^2\times\{\pt\}$ and then perform Kirby moves to show that the result is diffeomorphic to the result of performing $(0,2n,1)$ surgery along $T^2\times 0$. We note that the final step consists of a 2-/3-handle cancellation.}
    \label{fig:transformsparalleltori}
\end{figure}
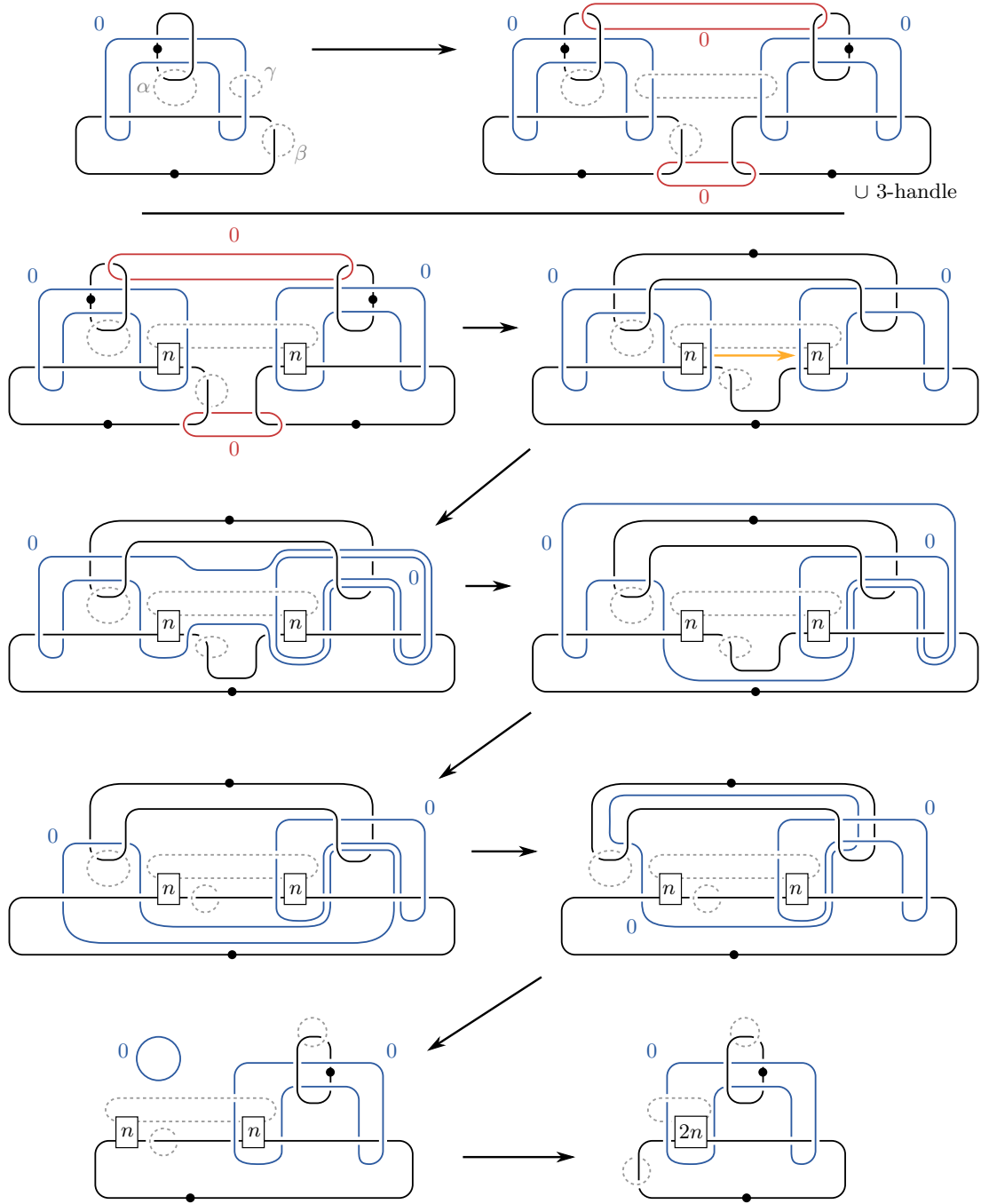

Finally, in Figure \ref{fig:transformsparalleltori}, we adapt the construction from \cite{akbulut:multiple} to show that performing $(0,n,1)$-surgeries on two parallel copies of $T'$ yields the same manifold up to diffeomorphism rel.\ boundary as performing one $(0,2n,1)$-surgery along $T'$.
\end{proof}

With this topological setup in place, we now distinguish the  surfaces' branched covers.

\begin{proposition}\label{prop:sw}
    The 4-manifolds $M_{(0,2n,1)}$ are all distinct up to diffeomorphism rel.\ boundary.
\end{proposition}

\begin{proof} We will construct a 4-dimensional cap $Q$ with $\partial Q = -\partial M$ and distinguish the closed 4-manifolds $Z_n = M_{(0,2n,1)} \cup Q$. To construct the cap, we begin by viewing $\partial M$ as $\partial M_{(0,1,0)}$ (Figure~\ref{fig:dbcsurgery}, middle). Attach a $(-1)$-framed 2-handle to $M_{(0,1,0)}$ along $\gamma$. As demonstrated in Figure~\ref{fig:infinite_handles_3}, the resulting 4-manifold admits the structure of a Stein domain \cite{gompf:stein}. By work of Lisca-Mati\'c \cite{lisca-matic:embed}, it embeds into a closed, minimal K\"ahler surface $Z$ that can be chosen to satisfy $b_2^+ >1$ and  $K_Z \cdot K_Z >0$ for the canonical class $K_Z=-c_1(Z)\in H^2(Z;\Z)$. Let $Q$ be the 4-manifold $Z \setminus \mathring{M}_{(0,1,0)}$.

\begin{figure}
       \labellist
\small\hair 2pt
\pinlabel ${M_{(0,1,0)} \cup_\gamma \text{\footnotesize 2-handle}}$ at 125 -30
\pinlabel {\large $=$} at 300 100
\pinlabel {\large $=$}  at  640 100

\endlabellist
    \includegraphics[width=\linewidth]{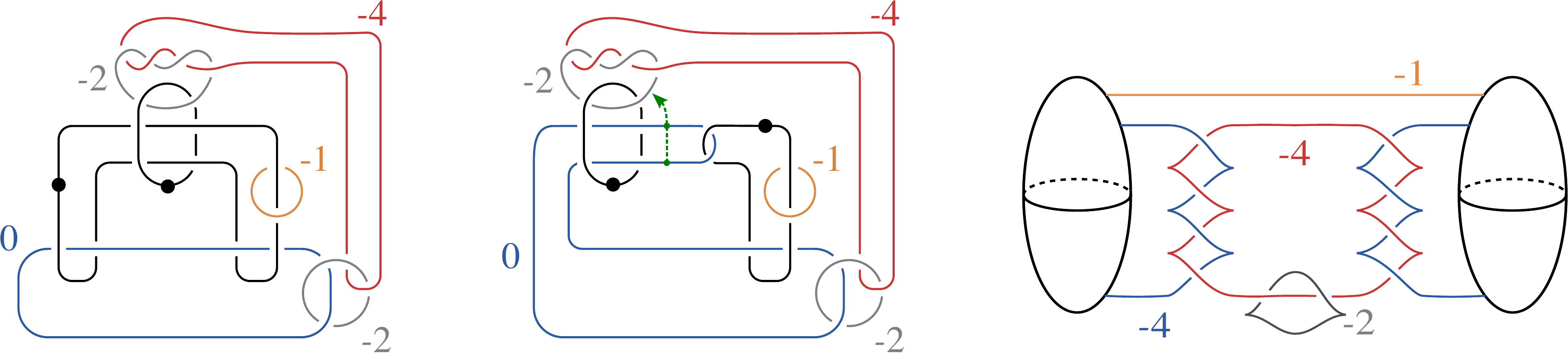}
    
    \bigskip
    \medskip
    \captionsetup{width=.875\linewidth}
    \caption{On the left, a handle diagram for the 4-manifold obtained by attaching a $(-1)$-framed 2-handle to $\gamma \subset \partial M_{(0,1,0)}$. The second diagram is obtained by isotopy. After performing the two handleslides indicated in the second diagram and canceling a 1-/2-handle pair, we obtain the (Stein) handle diagram on the right.}
    
    \label{fig:infinite_handles_3}
\end{figure}

Let $Z_0$ denote the 4-manifold $M \cup Q$, and extend this to an infinite family $Z_n = M_{(0,2n,1)} \cup Q$.  By construction, $Z_n$ is obtained from $Z_0$ by torus surgery of type $(0,2n,1)$, and $Z$ is obtained from $Z_0$ by surgery of type $(0,1,0)$. It follows from the Morgan-Mrowka-Szab\'o formula \cite[Theorem~1.1]{mms:torus} that the Seiberg-Witten invariants of $Z_n$ can be computed in terms of the invariants of $Z$, the $(0,1,0)$-surgery, and $Z_0$, the $(0,0,1)$-surgery. 

We collect some observations that simplify this calculation. First, we claim that $Z_0$ has vanishing Seiberg-Witten invariants. Recall that the construction of $Q$ involved attaching a $(-1)$-framed 2-handle to $\partial M$ along the curve $\gamma$. Since $\gamma$ bounds a smoothly embedded disk in $M$, it follows that $Z_0=M\cup Q$ contains a smoothly embedded 2-sphere of square $-1$. Moreover, since $\gamma$ is a meridian to the 0-framed 2-handle in $M$, this 2-sphere intersects the surgery torus $T^2 \times \{0\} \subset M$ transversely in a single point. It follows that the sum of these homology classes can be represented by a torus of square $+1$. This violates the adjunction inequality, so there cannot exist any Seiberg-Witten basic classes on $Z_0$. Turning to the K\"ahler surface $Z$,  note that the only Seiberg-Witten basic classes  are $\pm K_Z$ and that  $|\sw_Z(K_Z)|=1$ \cite{morgan:sw}.

To calculate the invariants of $Z_n$, let $\kappa \in H_2(Z;\Z)$ denote  the Poincar\'e dual of $K_Z$, and let $\kappa_n$ denote any class in $H_2(Z_n;\Z)$ such that the  $\operatorname{Spin}^c$ structures on $Z_n$ and $Z$ corresponding to $\kappa_n$ and $\kappa$, respectively, coincide where these 4-manifolds agree (i.e., away from the torus surgery neighborhood $T^2 \times D^2$). Let us also use $[T]$ in each of $H_2(Z_n;\Z)$ and $H_2(Z;\Z)$ to denote the dual torus \emph{after} surgery (i.e., the class of the core torus $T^2 \times \{0\}$ after regluing $T^2 \times D^2$ to perform the  torus surgery). With this notation in place,  the preceding observations combine with \cite[Theorem~1.1]{mms:torus} to imply that the only Seiberg-Witten basic classes on $Z_n$ are those of the form $\pm \kappa_n + i [T]$ for $ i \in \Z$, and that the sums of the Seiberg-Witten invariants of these classes satisfy
\begin{align*}
\sum_{i \in \Z} \sw_{Z_n}(\kappa_n + i [T]) &= 2n \sum_{i \in \Z} \sw_{Z}(\kappa+i [T])=2n \big( \sw_Z(\kappa) + 0 \big) = \pm 2n
\\
\sum_{i \in \Z} \sw_{Z_n}(-\kappa_n + i [T]) &= 2n \sum_{i \in \Z} \sw_{Z}(-\kappa+i [T])=2n \big( \sw_Z(-\kappa) + 0 \big) = \pm 2n,
\end{align*}
where the sums in the middle of each equation reduce to a single term because $\pm\kappa$ are the unique basic classes of $Z$ and must be linearly independent from $i[T]$ for $i \neq 0$ because $\kappa\cdot \kappa = K_Z \cdot K_Z >0$ but $i[T] \cdot i [T] =0$.

It follows that $Z_n$ is distinguished from $Z_m$ by its Seiberg-Witten invariants whenever $n \neq m$. Since $Z_n=M_{(0,2n,1)} \cup Q$ and $Z_m = M_{(0,2m,1)} \cup Q$, we conclude that $M_{(0,2n,1)}$ and $M_{(0,2m,1)}$ are distinct up to diffeomorphism rel boundary.
\end{proof}

\begin{proof}[Proof of Theorem \ref{thm:infinite}]
By Proposition \ref{prop:xpcover}, the 4-manifold $M_{(0,2n,1)}$ is the  branched cover of $B^4$ along $S_n$. Moreover, for all $m$, there is a natural identification of $\partial M_{(0,2n,1)}$ with $\partial M_{(0,2m,1)}$ arising from their identification with $\partial M = \Sigma_2(S^3,L)$. By Proposition~\ref{prop:sw}, these 4-manifolds are distinct up to diffeomorphism rel.\ boundary whenever $m \neq n$. It follows that $S_n$ and $S_m$  are not smoothly isotopic rel.\ boundary in $B^4$.
\end{proof}


\bibliographystyle{alpha} 
\bibliography{biblio}

\newcommand{\etalchar}[1]{$^{#1}$}
\begin{thebibliography}{MWW22}

\bibitem[AK80]{kirbyakbulut}
Selman Akbulut and Robion Kirby.
\newblock Branched covers of surfaces in {$4$}-manifolds.
\newblock {\em Math. Ann.}, 252(2):111--131, 1979/80.

\bibitem[Akb13]{akbulut:multiple}
Selman Akbulut.
\newblock Topology of multiple log transforms of 4-manifolds.
\newblock {\em Internat. J. Math.}, 24(7):1350052, 14, 2013.

\bibitem[Akb16]{akbulutbook}
Selman Akbulut.
\newblock {\em 4-manifolds}, volume~25 of {\em Oxford Graduate Texts in
  Mathematics}.
\newblock Oxford University Press, Oxford, 2016.

\bibitem[Alf70]{alford}
William~R. Alford.
\newblock Complements of minimal spanning surfaces of knots are not unique.
\newblock {\em Ann. of Math. (2)}, 91:419--424, 1970.

\bibitem[Alt12]{altman}
Irida Altman.
\newblock Sutured {F}loer homology distinguishes between {S}eifert surfaces.
\newblock {\em Topology Appl.}, 159(14):3143--3155, 2012.

\bibitem[Art26]{artin}
Emil Artin.
\newblock Zur isotopie zweidimensionalen fl\"achen im {R}$^4$.
\newblock {\em Abh. Math. Sem.}, pages 174--177, 1926.

\bibitem[AS70]{schaufele}
William~R. Alford and Christopher~B. Schaufele.
\newblock Complements of minimal spanning surfaces of knots are not unique.
  {II}.
\newblock In {\em Topology of {M}anifolds ({P}roc. {I}nst., {U}niv. of
  {G}eorgia, {A}thens, {G}a., 1969)}, pages pp 87--96. Markham, Chicago, Ill.,
  1970.

\bibitem[BN05]{barnatan}
Dror Bar-Natan.
\newblock Khovanov's homology for tangles and cobordisms.
\newblock {\em Geom. Topol.}, 9:1443--1499, 2005.

\bibitem[CDGW]{snappy}
Marc Culler, Nathan~M. Dunfield, Matthias Goerner, and Jeffrey~R. Weeks.
\newblock Snap{P}y, a computer program for studying the geometry and topology
  of $3$-manifolds.
\newblock \url{http://snappy.computop.org}.

\bibitem[CP20]{conway-powell}
Anthony Conway and Mark Powell.
\newblock Embedded surfaces with infinite cyclic knot group.
\newblock {\em Geom.\ Topol. (to appear)}, 2020.

\bibitem[Dai73]{daigle}
Roy~J. Daigle.
\newblock More on complements of minimal spanning surfaces.
\newblock {\em Rocky Mountain J. Math.}, 3:473--482, 1973.

\bibitem[Eis77]{eisner}
Julian~R. Eisner.
\newblock Knots with infinitely many minimal spanning surfaces.
\newblock {\em Trans. Amer. Math. Soc.}, 229:329--349, 1977.

\bibitem[FS97]{fintushel-stern:surfaces}
Ronald Fintushel and Ronald~J. Stern.
\newblock Surfaces in 4-manifolds.
\newblock {\em Math. Res. Lett.}, 4(6):907--914, 1997.

\bibitem[Gab86]{gabaidetecting}
David Gabai.
\newblock Detecting fibred links in {$S^3$}.
\newblock {\em Comment. Math. Helv.}, 61(4):519--555, 1986.

\bibitem[Gom98]{gompf:stein}
Robert~E. Gompf.
\newblock Handlebody construction of {S}tein surfaces.
\newblock {\em Ann. of Math. (2)}, 148(2):619--693, 1998.

\bibitem[GS99]{gompfstipsicz}
Robert~E. Gompf and Andr\'{a}s~I. Stipsicz.
\newblock {\em {$4$}-manifolds and {K}irby calculus}, volume~20 of {\em
  Graduate Studies in Mathematics}.
\newblock Amer.~Math.~Society, Providence, RI, 1999.

\bibitem[HJS13]{HJS}
Matthew Hedden, Andr\'{a}s Juh\'{a}sz, and Sucharit Sarkar.
\newblock On sutured {F}loer homology and the equivalence of {S}eifert
  surfaces.
\newblock {\em Algebr. Geom. Topol.}, 13(1):505--548, 2013.

\bibitem[HKM{\etalchar{+}}]{ancillary}
Kyle Hayden, Seungwon Kim, Maggie Miller, JungHwan Park, and Isaac Sundberg.
\newblock Ancillary files with the ar{X}iv version of ``{S}eifert surfaces in
  the 4-ball".

\bibitem[HS21]{hayden-sundberg}
Kyle Hayden and Isaac Sundberg.
\newblock Khovanov homology and exotic surfaces in the 4-ball.
\newblock {\em arXiv:2108.04810}, 2021.

\bibitem[HW92]{henry-weeks}
Shawn~R. Henry and Jeffrey~R. Weeks.
\newblock Symmetry groups of hyperbolic knots and links.
\newblock {\em J. Knot Theory Ramifications}, 1(2):185--201, 1992.

\bibitem[Jac04]{jacobsson}
Magnus Jacobsson.
\newblock An invariant of link cobordisms from {K}hovanov homology.
\newblock {\em Algebr. Geom. Topol.}, 4:1211--1251, 2004.

\bibitem[JM18]{juhaszmarengon}
Andr\'{a}s Juh\'{a}sz and Marco Marengon.
\newblock Computing cobordism maps in link {F}loer homology and the reduced
  {K}hovanov {TQFT}.
\newblock {\em Selecta Math. (N.S.)}, 24(2):1315--1390, 2018.

\bibitem[JPW14]{johnson-pelayo-wilson}
Jesse Johnson, Roberto Pelayo, and Robin Wilson.
\newblock The coarse geometry of the {K}akimizu complex.
\newblock {\em Algebr. Geom. Topol.}, 14(5):2549--2560, 2014.

\bibitem[Kak91]{kakimizu}
Osamu Kakimizu.
\newblock Doubled knots with infinitely many incompressible spanning surfaces.
\newblock {\em Bull. London Math. Soc.}, 23(3):300--302, 1991.

\bibitem[Kau72]{kauffman}
Louis~H. Kauffman.
\newblock {\em Cyclic {B}ranched {C}overs, {\emph{O($n$)}}-{A}ctions and
  {H}ypersurface {S}ingularities}.
\newblock ProQuest LLC, Ann Arbor, MI, 1972.
\newblock Thesis (Ph.D.)--Princeton University.

\bibitem[Kho00]{khovanov00}
Mikhail Khovanov.
\newblock A categorification of the {J}ones polynomial.
\newblock {\em Duke Math. J.}, 101(3):359--426, 2000.

\bibitem[Kho02]{khovanov02}
Mikhail Khovanov.
\newblock A functor-valued invariant of tangles.
\newblock {\em Algebraic \& Geometric Topology}, 2:665--741, 2002.

\bibitem[Kho06]{khovanov06}
Mikhail Khovanov.
\newblock An invariant of tangle cobordisms.
\newblock {\em Trans. Amer. Math. Soc.}, 358(1):315--327, 2006.

\bibitem[Kir78]{kirby-old}
Rob Kirby.
\newblock Problems in low dimensional manifold theory.
\newblock In {\em Algebraic and geometric topology ({P}roc. {S}ympos. {P}ure
  {M}ath., {S}tanford {U}niv., 1976), {P}art 2}, Proc. Sympos. Pure Math.,
  XXXII, pages 273--312. Amer. Math. Soc., Providence, R.I., 1978.

\bibitem[Kir97]{kirby}
Rob Kirby.
\newblock Problems in low-dimensional topology.
\newblock In William~H. Kazez, editor, {\em Geometric topology ({A}thens, {GA},
  1993)}, volume~2 of {\em AMS/IP Stud. Adv. Math.}, pages 35--473. Amer. Math.
  Soc., Providence, RI, 1997.

\bibitem[Kob89]{kobayashi}
Tsuyoshi Kobayashi.
\newblock Uniqueness of minimal genus {S}eifert surfaces for links.
\newblock {\em Topology Appl.}, 33(3):265--279, 1989.

\bibitem[Kob92]{kobayashipretzel}
Tsuyoshi Kobayashi.
\newblock A construction of {$3$}-manifolds whose homeomorphism classes of
  {H}eegaard splittings have polynomial growth.
\newblock {\em Osaka J. Math.}, 29(4):653--674, 1992.

\bibitem[Liv82]{livingston}
Charles Livingston.
\newblock Surfaces bounding the unlink.
\newblock {\em Michigan Math. J.}, 29(3):289--298, 1982.

\bibitem[LM97]{lisca-matic:embed}
Paolo Lisca and Gordana Mati\'{c}.
\newblock Tight contact structures and {S}eiberg-{W}itten invariants.
\newblock {\em Invent. Math.}, 129(3):509--525, 1997.

\bibitem[LS22]{lipshitz-sarkar}
Robert Lipshitz and Sucharit Sarkar.
\newblock A mixed invariant of nonorientable surfaces in equivariant {K}hovanov
  homology.
\newblock {\em Trans. Amer. Math. Soc.}, 375(12):8807--8849, 2022.

\bibitem[Lyo74]{lyon}
Herbert~C. Lyon.
\newblock Simple knots without unique minimal surfaces.
\newblock {\em Proc. Amer. Math. Soc.}, 43:449--454, 1974.

\bibitem[MMS97]{mms:torus}
John~W. Morgan, Tomasz~S. Mrowka, and Zolt\'{a}n Szab\'{o}.
\newblock Product formulas along {$T^3$} for {S}eiberg-{W}itten invariants.
\newblock {\em Math. Res. Lett.}, 4(6):915--929, 1997.

\bibitem[Mor96]{morgan:sw}
John~W. Morgan.
\newblock {\em The {S}eiberg-{W}itten equations and applications to the
  topology of smooth four-manifolds}, volume~44 of {\em Mathematical Notes}.
\newblock Princeton University Press, Princeton, NJ, 1996.

\bibitem[MWW22]{morrison-walker-wedrich}
Scott Morrison, Kevin Walker, and Paul Wedrich.
\newblock Invariants of 4--manifolds from {K}hovanov--{R}ozansky link homology.
\newblock {\em Geom. Topol.}, 26(8):3367--3420, 2022.

\bibitem[Par78]{parris}
Richard~L. Parris.
\newblock {\em {P}retzel {K}nots}.
\newblock ProQuest LLC, Ann Arbor, MI, 1978.
\newblock Thesis (Ph.D.)--Princeton University.

\bibitem[Pla06]{plamenevskaya:transverse-Kh}
Olga Plamenevskaya.
\newblock Transverse knots and {K}hovanov homology.
\newblock {\em Math. Res. Lett.}, 13(4):571--586, 2006.

\bibitem[Rol76]{rolfsen}
Dale Rolfsen.
\newblock {\em Knots and links}.
\newblock Mathematics Lecture Series, No. 7. Publish or Perish, Inc., Berkeley,
  Calif., 1976.

\bibitem[Rud90]{rudolph:kauffman-bound}
Lee Rudolph.
\newblock A congruence between link polynomials.
\newblock {\em Math. Proc. Cambridge Philos. Soc.}, 107(2):319--327, 1990.

\bibitem[SS64]{schubert-soltsien}
Horst Schubert and Kay Soltsien.
\newblock Isotopie von {F}l\"{a}chen in einfachen {K}noten.
\newblock {\em Abh. Math. Sem. Univ. Hamburg}, 27:116--123, 1964.

\bibitem[SS23]{sundberg-swann}
Isaac Sundberg and Jonah Swann.
\newblock Relative {K}hovanov-{J}acobsson classes.
\newblock {\em Algebr.\ Geom.\ Topol.}, 22(8):3983--4008, 2023.

\bibitem[Sza19]{hfk-calc}
Zolt{\'a}n Szab{\'o}.
\newblock {K}not {F}loer homology calculator.
\newblock Available at
  \url{https://web.math.princeton.edu/~szabo/HFKcalc.html}, 2019.

\bibitem[{The}19]{sagemath}
{The Sage Developers}.
\newblock {S}agemath, the {S}age mathematics software system.
\newblock Available at \url{https://www.sagemath.org}, 2019.

\bibitem[Tro75]{trotter}
Hale~F. Trotter.
\newblock Some knots spanned by more than one knotted surface of minimal genus.
\newblock In {\em Knots, groups, and {$3$}-manifolds (papers dedicated to the
  memory of {R}. {H}. {F}ox)}, pages 51--62. Ann. of Math. Studies, No. 84.
  1975.

\bibitem[Vaf15]{vafaee}
Faramarz Vafaee.
\newblock Seifert surfaces distinguished by sutured {F}loer homology but not
  its {E}uler characteristic.
\newblock {\em Topology and its Applications}, 184:72--86, 2015.

\bibitem[Wan22]{josh}
Joshua Wang.
\newblock The cosmetic crossing conjecture for split links.
\newblock {\em Geom. Topol.}, 26(7):2941--3053, 2022.

\bibitem[Wil08]{wilson}
Robin~T. Wilson.
\newblock Knots with infinitely many incompressible {S}eifert surfaces.
\newblock {\em J. Knot Theory Ramifications}, 17(5):537--551, 2008.

\bibitem[Zee65]{zeeman}
Erik~Christopher Zeeman.
\newblock Twisting spun knots.
\newblock {\em Trans. Am. Math. Soc.}, 115:471--495, 1965.

\end{thebibliography}


\end{document}